\documentclass[a4paper,11pt]{article}
\usepackage[top=3.0cm, bottom=3.0cm, inner=3.0cm, outer=3.0cm, includefoot]{geometry}

\usepackage[utf8]{inputenc}
\usepackage[T1]{fontenc}
\usepackage{authblk}

\usepackage{verbatim}
\usepackage{multirow}
\usepackage{multicol}

\usepackage{amssymb}
\usepackage{amsmath}
\usepackage{mathtools,bbm}
\usepackage{graphicx,tikz}
\usepackage{amsthm}
\usepackage{caption,subcaption}

\usepackage{color}
\usepackage{enumerate}

\usepackage{cancel}

\newcommand{\eChar}{\begin{enumerate}[(i)]}
\newcommand{\eCharR}{\begin{enumerate}[(a)]}
\newcommand{\eBr}{\begin{enumerate}[(1)]}

\newcommand{\Abstract}

\title
{Vertex isoperimetry on signed graphs and spectra of non-bipartite Cayley and Cayley sum graphs
}

\author[1]{Chunyang Hu}\author[2]{Shiping Liu}
\affil[1,2]{School of Mathematical Sciences, University of Science and Technology of China, Hefei 230026, China}
\affil[1]{chunyanghu@mail.ustc.edu.cn}
\affil[2]{spliu@ustc.edu.cn}

\date{\today}

\date{\today}

\theoremstyle{plain}
\newtheorem{lemma}{Lemma}[section]
\newtheorem{theorem}[lemma]{Theorem}
\newtheorem{proposition}[lemma]{Proposition}

\theoremstyle{definition}

\newtheorem{definition}[lemma]{Definition}
\newtheorem{remark}[lemma]{Remark}

\newtheorem{example}[lemma]{Example}

\numberwithin{equation}{section}

\begin{document}

\maketitle

\pagestyle{plain}

\begin{abstract}
  For a non-bipartite finite Cayley graph, we show the non-trivial eigenvalues of its normalized adjacency matrix lie in the interval $$\left[-1+\frac{ch_{out}^2}{d},1-\frac{Ch_{out}^2}{d}\right],$$ for some absolute constants $c$ and $C$, where $h_{out}$ stands for the outer vertex boundary isoperimetric constant. This improves upon recent obtained estimates aiming at a quantitative version of a result due to Breuillard, Green, Guralnick and Tao. 
  We achieve this by extending the work of Bobkov, Houdr\'e and Tetali on vertex isoperimetry to the setting of signed graphs. 
  We further extend our interval estimate to the settings of vertex transitive graphs and Cayley sum graphs. As a byproduct, we answer positively open questions proposed recently by Moorman, Ralli and Tetali.
\end{abstract}

\section{Introduction}

We consider the spectral theory of finite connected graphs in this article. Let $G=(V,E)$ be a finite graph. Consider its normalized adjacency matrix $D^{-1}A$ and normalized Laplacian $I-D^{-1}A$. Here we denote by $A$ the adjacency matrix, $D$ the diagonal matrix given by vertex degrees, and $I$ the identity matrix of size $N$ by $N$ with $N=|V|$. Let us list the eigenvalues counting multiplicities of the two matrices, respectively, as below:
\begin{align*}
& t_1\leq t_2\leq \cdots\leq t_{N-1}\leq t_N,\\
& \lambda_1\leq \lambda_2\leq \cdots\leq \lambda_{N-1}\leq \lambda_N.
\end{align*}
Observe that $\lambda_k=1-t_{N-k+1}$.
It is well known that the eigenvalues of the normalized adjacency matrix lie in the interval $[-1,1]$ and, consequently, the eigenvalues of the normalized Laplacian matrix lie in $[0,2]$. Notice that $t_N=1$. Moreover, we have $t_{N-1}=1$ if and only if the graph is disconnected, and $t_1=-1$ if and only if it is bipartite.

Eigenvalue estimates of graphs in terms of the \emph{vertex isoperimetric constant} $h_{out}$ have been first studied by Alon \cite{Alon} in 1985, inspired by the work of Cheeger \cite{Cheeger} on Riemannian manifolds. Notice that the vertex isoperimetric constant $h_{out}$ was also called \emph{magnification} \cite{Alon} or \emph{vertex expansion}, see, e.g., \cite{MRT}.
Estimates of graph eigenvalues via the \emph{Cheeger constant} $h$ (also known as \emph{edge expansion} or \emph{isoperimetric constant}) have been extensively studied, see, e.g., \cite{AM, BKW, Chung, D, LS, LOT, Miclo, Miclo2, Mohar, SJ}. 
In contrast to the Cheeger constant $h$ defined via \emph{edge boundaries} (see (\ref{eq:edgebdy}) and (\ref{eq:Cheegerhhout}) below), the vertex isoperimetric constant $h_{out}$ is based on \emph{vertex boundaries} (see (\ref{eq:outbdy}) and (\ref{eq:Cheeger_hout}) below). 
Using the techniques from linear algebra and max-flow min-cut theorem, Alon \cite{Alon} (see also \cite[Section 2.4]{Chung}) showed that 
\begin{equation}\label{eq:Alon}
    \lambda_2=1-t_{N-1}\geq \frac{h_{out}^2}{2d_{\max}(2+h_{out})},
\end{equation}
where $d_{\max}$ stands for the maximal vertex degree.

Consider a Cayley graph $C(X,S)$ of a finite group $X$ with respect to a set $S\subset X$, where $S$ is closed under taking inverses and does not include the identity element.
In particular, $C(X,S)$ is a regular graph.
In a recent work, Breuillard, Green, Guralnick, and Tao \cite[Proposition E.1]{BGGT} observed the following special feature of Cayley graphs which does not hold for arbitrary regular graphs: if a non-bipartite Cayley graph is an expander graph in the sense that $h_{out}$ is bounded away from $0$, then the spectrum of the normalized adjacency matrix except the trivial one $t_N=1$ is not only bounded away from $1$ but also from $-1$. Building upon the arguments in \cite[Proposition E.1]{BGGT}, Biswas \cite{Biswas_EJC} established a quantitative version of this observation: For a $d$-regular non-bipartite Cayley graph, the eigenvalues $\{t_i\}_{i=1}^{N-1}$ of the normalized adjacency matrix lie in the interval 
\[\left[-1+\frac{h_{out}^4}{2^9d^6(d+1)^2}, 1-\frac{h_{out}^2}{2d^2}\right].\] 
Later, Biswas and Saha \cite{BS_AC} improved it to be
\[\left(-1+\frac{h_{out}^4}{2^9d^8}, 1-\frac{h_{out}^2}{2d^2}\right].\] 
In fact, the upper bound of the above intervals can be improved by Alon's estimate (\ref{eq:Alon}).
Since $h_{out}\leq 2$, we derive from (\ref{eq:Alon}) that the upper bound of the interval can be improved to be
\begin{equation}\label{eq:Alonhoutsquare}
  1-\frac{h_{out}^2}{2d(2+h_{out})}\leq 1-\frac{Ch_{out}^2}{d}  
\end{equation}
for an absolute constant $C$.
The same improvement as (\ref{eq:Alonhoutsquare}) can be achieved by a companion result due to  
Bobkov, Houdr\'e and Tetali \cite{BHT} telling that
\begin{equation}\label{eq:BHT}
    \lambda_2=1-t_{N-1}\geq\frac{\left(\sqrt{1+h_{out}}-1\right)^2}{4d}.
\end{equation}
The proof of (\ref{eq:BHT}) is built upon a functional analytic method involving a Poincar\'e-type functional constant $\lambda_{\infty}$. 

Employing the dual Cheeger inequality due to Bauer-Jost \cite{BJ} and Trevisan \cite{Trevisan}, Moorman, Ralli, and Tetali \cite{MRT} further improved the previous estimates to be 
\begin{equation*}
    \left[ -1+\frac{ch_{out}^2}{d^2},1-\frac{Ch_{out}^2}{d}\right],
\end{equation*}
for some absolute constants $c$ and $C$. Considering the cycle graphs, the estimate above is tight up to a factor depending on $d$, see \cite[Example 2.7]{MRT} and \cite[Proposition 7.4]{Liu}.

In this article, we improve the lower bound of the interval as follows.
\begin{theorem}\label{thm:Main1}
For a non-bipartite finite Cayley graph with degree $d$, the eigenvalues of the corresponding  normalized adjacency matrix except the largest one lie in the interval 
 \begin{equation}\label{eq:resultA}
   \left[ -1+\frac{ch_{out}^2}{d},1-\frac{Ch_{out}^2}{d}\right], 
\end{equation}   
for some absolute constant $c$ and $C$. 
\end{theorem}



That is, we improve the order of $d$ in the lower bound of the interval. In our result, the lower bound of the interval matches nicely with the upper bound. However, we do not know whether the order of $d$ is tight or not. Since the eigenvalues of a bipartite graph are symmetric about $0$, our result implies that any eigenvalue which is not equal to $1$ or $-1$ of a finite Cayley graph lies in the interval (\ref{eq:resultA}).
 
We achieve this by extending the work of Bobkov, Houdr\'e and Tetali \cite{BHT} to signed graphs. A signed graph \cite{Harary, Z} $(G,\sigma)$ is a graph $G=(V,E)$ equipped with a signature $\sigma: E\to \{+1,-1\}$. The corresponding normalized signed Laplacian matrix $\Delta^{\sigma}$ is given by $I-D^{-1}A^{\sigma}$, where $A^\sigma$ stands for the signed adjacency matrix, i.e., the $(x,y)$-entry of $A$ is given by $\sigma_{xy}:=\sigma(\{x,y\})$ whenever $x\sim y$. 

We list the eigenvalues of the matrix $\Delta^\sigma$ counting multiplicities as follows:
\[\lambda^{\sigma}_1\leq \lambda^{\sigma}_2\leq \cdots\leq \lambda^{\sigma}_{N-1}\leq \lambda^{\sigma}_N.\]
The spectral theory of signed graphs have been well studied, see, e.g., \cite{AL, BSS, BGKLM, LLPP, LMP, SW} and the survey \cite{BCKW}. 

In order to extend the work of Bobkov, Houdr\'e and Tetali \cite{BHT}, we first introduce the following two concepts on signed graphs. One is a signed version $h_{out}^\sigma$ of the vertex isoperimetric constant (see Definition \ref{def:beta}), which is inspired by the works on signed Cheeger constants \cite{AL,BSS, LLPP}. The other one is a signed version $\lambda_{\infty}^{\sigma}$ of the Poincar\'e-type functional constant \`{a} la Bobkov, Houdr\'e and Tetali (see Definition \ref{def:BHT}). Combining the methods from \cite[Section 3]{BHT} and the proof of signed Cheeger inequalities in \cite{AL, LLPP}, we establish the following estimate.
\begin{theorem} \label{thm: Main2}
For a $d$-regular signed graph $(G,\sigma)$, we have 
    \begin{equation}\label{eq:key}
    2d\lambda_1^\sigma\geq \lambda_{\infty}^{\sigma}\geq \left(\sqrt{1+h_{out}^\sigma}-1\right)^2.
\end{equation}
\end{theorem}

  The estimate (\ref{eq:key}) is tight up to a factor depending on $d$, which can be seen from the curvature consideration, see Remark \ref{rem:buser}.
  
  While Theorem \ref{thm: Main2} plays a key role in the proof of Theorem \ref{thm:Main1}, it is of its own interest. Indeed, a consequence of Theorem \ref{thm: Main2} tells for any $d$-regular graph $G$ that

\begin{equation}\label{eq:intro1}
    2-\lambda_N=1+t_1\geq \frac{\left(\sqrt{1+\beta_{out}}-1\right)^2}{2d},
\end{equation}
where $\beta_{out}$ is the so-called \emph{outer vertex bipartiteness constant} \cite[Definition 2.1]{MRT} (see also (\ref{eq:betaout})). We refer to Proposition \ref{prop:h=beta} for the close relation between the constants $h_{out}^\sigma$ and $\beta_{out}$.
The estimate (\ref{eq:intro1}) actually answers an open question of Moormann, Ralli and Tetali \cite[Section 3]{MRT} positively. Indeed, our approach is inspired by their open question asking whether Bobkov, Houdr\'e and Tetali's strategy can be applied to the outer vertex bipartiteness constant $\beta_{out}$ or not.

We further extend our Theorem \ref{thm:Main1} to vertex transitive graphs, and even certain class of non-vertex-transitive graphs. 

\begin{theorem}\label{thm:Main3}
For a non-bipartite finite vertex transitive graph, or a non-bipartite finite Cayley sum graph, the eigenvalues of the corresponding  normalized adjacency matrix except the largest one lie in the interval (\ref{eq:resultA})
for some absolute constants $c$ and $C$. 
\end{theorem}
We refer to Section \ref{section:Cayley} for the differences between Cayley and Cayley sum graphs. In particular,  Cayley sum graphs need not be vertex transitive. It has been asked by Moorman, Ralli and Tetali in \cite{MRT} as an open question whether it is possible to extend their result \cite[Theorem 2.6]{MRT} to the setting of Cayley sum graphs. Our Theorem \ref{thm:Main3} answers positively this open question and improves the related estimates in \cite{BS_AC,BS_2021,Saha}. Our proof is built upon Theorem \ref{thm: Main2} and the proof method of Moorman-Ralli-Tetali \cite[Theorem 2.4]{MRT} modified with key ideas from Biswas-Saha \cite{BS_2021} and a very recent preprint of Saha \cite{Saha}.

Our Theorem \ref{thm: Main2} is also extendible. In fact, we show an eigenvalue estimate similar to (\ref{eq:key}) in terms of a vertex isoperimetric constants  defined by both \emph{inner and outer boundaries} on signed graphs. Furthermore, we extend our results to more general connection graphs, with new introduced Cheeger type constants closely related to the work of Bandeira, Singer and Spielman \cite{BSS} and \cite{LLPP,LMP}.

The paper is structured as below. We collect preliminaries about Cayley and Cayley sum graphs, the vertex bipartiteness constant and the Cheeger constant on signed graphs in Section \ref{preliminary}. In Section \ref{new_concept}, we introduce new constants on signed graphs, including outer and symmetric vertex isoperimetric constants and a Poincar\'e-type functional constant. In Section \ref{signed_graph}, we prove the corresponding Cheeger-type inequalities on signed graphs and derive Theorem \ref{thm: Main2} as a corollary. In Section \ref{Cayley_graph}, we prove our main Theorem \ref{thm:Main1} for Cayley graphs and Theorem \ref{thm:Main3} for vertex transitive graphs and Cayley sum graphs.  In Section \ref{extensions}, we extend our Cheeger-type eigenvalue estimates to connection graphs.
\section{Preliminaries}\label{preliminary}
In this section, we collect preliminaries on Cayley graphs, Cayley sum graphs, vertex bipartiteness constant, signed graphs and signed Cheeger constants.

Let $G=(V,E)$ be an undirected finite graph with a vertex set $V$ and an edge set $E$, where $E$ might contain self-loops, i.e., edges connecting a vertex to itself. 
We write $x\sim y$ if $\{x,y\}\in E$. In particular, we have $x\sim x$ whenever there is a self-loop at $x$. For a subset $A$ of either $V$ or $E$, we denote its cardinality by $|A|$. The degree of a vertex $x$ is given by $d_{x}=|\{y\in V \,|\, y \sim x\}|$. 
Notice that, a self-loop is counted once here.
The volume of a subset $V_{1}\subseteq V$ is defined as $\mathrm{vol}(V_{1}):=\sum_{x\in V_{1}}d_{x}.$ A $d$-regular graph is a graph with all vertices having the same degree $d$. 

\subsection{Cayley graphs and Cayley sum graphs}\label{section:Cayley}
We compare some basic aspects of Cayley graphs and Cayley sum graphs in this subsection. While Cayley graphs have been studied extensively, much less results are known for Cayley sum graphs. Cayley sum graphs, also called addition Cayley graphs, addition graphs or sum graphs, were first explored for abelian groups \cite{Alon2, CGW, Chung2, Green, GLS}, and later extended to arbitrary groups \cite{AT,BS_AC}. 

We first recall their definitions. A subset $S$ of a group $X$ is called \emph{symmetric} if it is closed under taking inverses, that is, $x^{-1}\in S$ for any $x\in S$; $S$ is called \emph{normal} if it is closed under taking conjugation, that is, $xSx^{-1}=S$ holds for any $x\in X$.    
\begin{definition}\label{defn:CayleyCayleysum}
    Let $X$ be a group and $S$ be a set of elements in $X$. 
    \begin{itemize}
        \item A Cayley graph $C(X,S)$ of the group $X$ with respect to a symmetric set $S$ is a graph with vertex set $X$ such that any two vertices $x$ and $y$ are adjacent if and only if $xy^{-1}\in S$. 
        \item A Cayley sum graph $C_{\Sigma}(X,S)$ of the group $X$ with respect to a normal set $S$ is a graph with vertex set $X$ such that any two vertices $x$ and $y$ are adjacent if and only if $xy\in S$. 
    \end{itemize}
\end{definition}
Notice that the restrictions on the set $S$ in Definition \ref{defn:CayleyCayleysum} are sufficient and necessary to ensure the resulting graphs to be undirected, see \cite[Section 1.3.18]{BM} for Cayley graphs and \cite[Lemma 2.6]{BS_AC} for Cayley sum graphs.

Although the definitions of Cayley and Cayley sum graphs look similar, their properties are quite different.

A Cayley graph $C(X,S)$ does not contain self-loops if and only if $\mathrm{id}_X\notin S$. In particular, self-loops appear either at every vertex or at no vertex of a Cayley graph. Hence, the usual definition of a Cayley graph \cite[Section 1.3.18]{BM} requires further that $\mathrm{id}_X\notin S$. On the other hand, a Cayley sum graph $C_\Sigma(X,S)$ does not contain self-loops if and only if $S$ is square free, i.e., there is no $x\in X$ such that $x^2\in S$. It can happen that only some of the vertices of a Cayley sum graph possess self-loops. For convenience, we allow self-loops for both Cayely and Cayley sum graphs in this paper.  

A Cayley graph $C(X,S)$ is connected if and only if the set $S$ is a generating set of the group $X$. The connectivity of a Cayley sum graph $C_{\Sigma}(X,S)$ has been characterized in \cite[Theorems 1 and 2]{AT}. In particular, it is proved that a Cayley sum graph $C_{\Sigma}(X,S)$ of a finite group $X$ is connected if and only if $S$ is a generating set and the subgroup $\langle S^{-1}S\rangle$ generated by $S^{-1}S$ has index no larger than $2$. This is a generalization of the earlier result \cite[Proposition 2.3]{CGW} for finite abelian groups.

A Cayley graph $C(X,S)$ is non-bipartite if and only if the group $X$ has no subgroup of index $2$ which is disjoint from the set $S$, see, e.g., \cite[Proposition E.1]{BGGT} and \cite[Proposition 2.6]{Biswas_EJC}. For Cayley sum graphs, it is observed in \cite[Proposition 2.5]{BS_AC} the following sufficient condition for bipartiteness: If the group $X$ contains a subgroup of index two which does not intersect the set $S$, then
the Cayley sum graph $C_\Sigma(X,S)$ is bipartite. We refer to \cite{CGW} for characterizations of bipartiteness for Cayley sum graphs of finite abelian groups. Actually, it is proved in \cite[Proposition 4.4]{CGW} that any connected bipartite Cayley sum graph of a finite abelian group, which has no self-loops, is a Cayley graph.

The structure of a Cayley graph $C(X,S)$ and a Cayley sum graph $C_{\Sigma}(X,S)$, even of the the same finite group $X$ and set $S$, can be very different.
\begin{example}\label{examplecc}
    Let $X=\mathbb{Z}_{n}:=\mathbb{Z}/n\mathbb{Z}$ with $n$ odd and $S=\{+1,-1\}$. Then $S$ is symmetric and normal. The Cayley graph $C(\mathbb{Z}_{n},\{+1,-1\}\})$ is an $n$-cycle, while the Cayley sum graph $C_{\Sigma}(\mathbb{Z}_{n},\{+1,-1\})$ is a path with two self-loops at its two end vertices, respectively. 
\end{example}

\subsection{Vertex-transitive graphs}

Let's recall the definition of vertex-transitive graph here.
\begin{definition}
     A graph $G=(V,E)$ is called vertex-transitive if its automorphism group $\mathrm{Aut}(G)$ acts transitively on $V$. That is, for any two vertices $x,y\in V$, there is an automorphism of $G$ mapping $x$ to $y$.
\end{definition}

Fix a vertex $v\in V$, its stabilizer subgroup $G_{v}$ is defined as
$$G_{v}:=\{g\in \mathrm{Aut}(G)|g(v)=v\}.$$
That is, $G_{v}$ is the subgroup consist of all the automotphisms in $\mathrm{Aut}(G)$ that map the vertex $v$ to itself. In a vertex transitive graph, all stabilizer subgroups $G_v, v\in V$ are conjugate in $\mathrm{Aut}(G)$  and hence, in particular, have the same size. The index of $G_v$ in $\mathrm{Aut}(G)$ is given by $|\mathrm{Aut}(G):G_v|=|\mathrm{Aut}(G)|/|G_v|=|V|$, see, e.g.,  \cite[Chapter 16]{Biggs}.

All Cayley graphs $C(X,S)$ are vertex-transitive. However, the converse is not true. For example, the Petersen graph is vertex-transitive but not Cayley. 

A Cayley sum graph $C_{\Sigma}(X,S)$ may not be vertex-transitive. For example, the graph $C_{\Sigma}(\mathbb{Z}_{n},\{+1,-1\})$ in Example \ref{examplecc} is not vertex-transitive, since two vertices have self-loops while the others do not have. For examples of non-vertex-transitive Cayley sum graphs without self-loops, we refer to \cite[Remark 2.8]{CGW} and 
\cite[Example 2.3]{ZY}.

\subsection{Cheeger constant and vertex isoperimetry constants}
For a subset $V_1\subseteq V$, its edge boundary $\partial (V_1)$ is defined as
\begin{equation}
    \partial (V_1)=\{\{x,y\}\in E \,|\, x\in V_{1},\,\,y\notin V_{1}\},\label{eq:edgebdy}
\end{equation}
while its outer vertex boundary $\partial_{out}(V_{1})$ and inner vertex boundary $\partial_{in}(V_{1})$ are defined as 
\begin{equation}
    \partial_{out}(V_{1})=\{y\in V\setminus V_{1} \,|\, \exists\, x\in V_{1}, x\sim y\}\label{eq:outbdy}
\end{equation} and  
\begin{equation}
\partial_{in}(V_{1})=\{x\in V_{1}\,|\,\exists\, y\in V\setminus V_{1},y\sim x\},\label{eq:inbdy}
\end{equation}
respectively. Moreover, the symmetric vertex boudary $\partial_S(V_1)$ is defined via 
\[\partial_S(V_1)=\partial_{in}(V_1)\cup \partial_{out}(V_1).\]
Clearly, we have $\partial_S(V_1)=\partial_S(V\setminus V_1)$.

For a $d$-regular graph, the Cheeger constant $h$ and the outer vertex isoperimetric constant $h_{out}$ is given by
\begin{equation}\label{eq:Cheegerhhout}
   h:=\min_{\substack{\emptyset\neq V_1\subseteq V\\\mathrm{vol}(V_1)\leq \frac{1}{2}\mathrm{vol}(V)}}\frac{|\partial(V_1)|}{\mathrm{vol}(V_1)}=\min_{\substack{\emptyset\neq V_1\subseteq V\\|V_1|\leq \frac{1}{2}|V|}}\frac{|\partial(V_1)|}{d|V_1|},
\end{equation}
and
\begin{equation}\label{eq:Cheeger_hout}
h_{out}:=\min_{\substack{\emptyset\neq V_1\subseteq V\\|V_1|\leq \frac{1}{2}|V|}}\frac{|\partial_{out}(V_1)|}{|V_1|}.
\end{equation}
It is straightforward to check that $h_{out}\geq h\geq \frac{1}{d}h_{out}$. 


\subsection{Vertex bipartiteness constant}

For a $d$-regular graph $G=(V,E)$, Trevisan \cite{Trevisan} defines a \emph{bipartiteness constant}
\begin{equation}\label{eq:Trevisan}
    \beta=\min_{\substack{L,R\subseteq V\\L\cap R=\emptyset,L\cup R\neq \emptyset}}\frac{e(L,L)+e(R,R)+|\partial(L\cup R)|}{d|L\cup R|},
\end{equation}
where $e(L,L):=\sum_{x\in L}\sum_{y\in L}a_{xy}$ with $a_{xy}=1$ if $x$ and $y$ are adjacent, and $a_{xy}=0$ otherwise.
Observe that we count in $E(L,L)$ every edge twice. Notice that $\beta=0$ if and only if $G$ has a bipartite connected component. Indeed, Trevisan \cite{Trevisan}, independently Bauer and Jost \cite{BJ}, show that 
\begin{equation}\label{eq:TBJ}
    \frac{\beta^2}{2}\leq 2-\lambda_N\leq 2\beta.
\end{equation} 
We note that Bauer and Jost introduce a so-called \emph{dual Cheeger constant} which equals $1-\beta$. Actually, Bauer and Jost explicitly carried out their results on graphs allowing multiple edges and multiple self-loops.

 Motivated by the work of Trevisan and Bauer-Jost, Moorman, Ralli and Tetali \cite{MRT} introduce the \emph{outer vertex bipartiteness constant} $\beta_{out}$ as follows:
\begin{equation}\label{eq:betaout}
    \beta_{out}=\min_{\substack{L,R\subseteq V\\L\cap R=\emptyset,L\cup R\neq \emptyset}}\frac{I(L)+I(R)+|\partial_{out}(L\cup R)|}{|L\cup R|},
\end{equation}
where $I(L)$ is the number of vertices in $L$ with a neighbor also in $L$, that is, 
\[I(L)=|\{x\in L\,|\,\exists \, y\in L,\,y\sim x\}|.\]
Observe that 
\begin{equation}\label{eq:betaoutbound}
    \beta_{out}\leq \min_{\substack{L,R\subseteq V\\L\cap R=\emptyset,L\cup R=V}}\frac{I(L)+I(R)}{|L\cup R|}\leq 1.
\end{equation}
Moreover, it is direct to check \cite[Theorem 2.3]{MRT} for a $d$-regular graph that
\begin{equation}\label{eq:betaoutbetadbetaout}
    \beta_{out}\geq \beta\geq \frac{1}{d}\beta_{out}.
\end{equation}

\subsection{Signed graphs}
A signed graph $(G,\sigma)$ is a finite graph $G=(V,E)$ equipped with a signature function $\sigma:E\rightarrow \{+1,-1\}$. 
To switch the signature $\sigma$ by a function $\tau:V\rightarrow \{+1,-1\}$ is to change the signature on each edge $x\sim y$ into
$$\sigma^{\tau}_{xy}=\tau(x)\sigma_{xy}\tau(y).$$
 And $\tau$ is called a \emph{switching function}. 
 Note that the signature of a self-loop is unchanged. 
 Two signatures $\sigma$ and $\sigma'$ are said to be \emph{switching equivalent} if there exists a switching function $\tau$ such that $\sigma'=\sigma^{\tau}$. We call a simple closed path in $(G,\sigma)$ a cycle. A cycle can be a self-loop, a triangle, etc. The signature of a cycle in $(G,\sigma)$ is defined as the multiplication of signatures on all edges of the cycle. A signed graph is called \emph{balanced} if all the signatures of its cycles are $+1$ \cite{Harary} (see also \cite[Section 2]{Z}). The following lemma provides an equivalent definition of balancedness via switching. 
\begin{lemma}[Zaslavsky \cite{Z}]\label{lem:Z}
A signed graph $(G,\sigma)$ is balanced if and only if it is switching equivalent to the all $+1$ signature.
\end{lemma}
For any function $f:V\to \mathbb{R}$, we define the corresponding normalized Laplacian $\Delta^\sigma$ as 
\begin{equation}\label{eq:signedLaplacian}
    \Delta^\sigma f(x)=\frac{1}{d_x}\sum_{y:y\sim x}(f(x)-\sigma_{xy}f(y)),\,\,\text{for any}\,\,x\in V.
\end{equation}
The normalized signed adjacency operator $D^{-1}A^\sigma$ is given by
\begin{equation}\label{eq:signedAdj}
    D^{-1}A^\sigma f(x)=f(x)-\Delta^\sigma f(x),\,\,\text{for any}\,\,x\in V.
\end{equation}
The normalized Laplacian $\Delta^\sigma$ has $N=|V|$ real eigenvalues lying in the interval $[0,2]$, listed with multiplicity as 
\[\lambda_1^\sigma \leq \cdots\leq \lambda^\sigma_N.\]
It is well known that $\lambda_1^\sigma=0$ if and only if $(G,\sigma)$ has a balanced connected component; $\lambda^\sigma_N=2$ if and only if $(G,-\sigma)$ has a balanced connected component, see, e.g., \cite{AL}.

\subsection{Signed Cheeger constant}
We recall the signed Cheeger constant introduced in \cite{AL,LLPP}, see also \cite{BSS}.
We use the following frustration index to quantify how far a signed graph or subgraph is from being balanced. For a subset $V_1\subseteq V$, define 
\begin{equation}\label{eq:edgeiota}
    \iota^\sigma(V_1)=\min_{\tau:V_{1}\to \{+1,-1\}}\frac{1}{2}\sum_{x\in V_1}\sum_{y\in V_{1},y\sim x}|\tau(x)-\sigma_{xy}\tau(y)|.
\end{equation}
Notice that $|\tau(x)-\sigma_{xy}\tau(y)|=|1-\sigma^\tau_{xy}|$. 
According to Lemma \ref{lem:Z}, $\iota^{\sigma}(V_{1})=0$ if and only if the induced signed subgraph of $V_{1}$ is balanced.

The signed Cheeger constant is defined as below.
\begin{definition}[\cite{AL,LLPP}]
Let $(G,\sigma)$ be a signed graph. We define the signed Cheeger constant $h^\sigma$ as
\[h^\sigma=\min_{\emptyset\neq V_1\subseteq V}\frac{\iota^\sigma(V_1)+|\partial (V_1)|}{\mathrm{vol}(V_1)}.\]
\end{definition}
Observe that $h^\sigma=0$ if and only if $(G,\sigma)$ has a balanced connected component. Indeed, it is shown in Atay-Liu \cite{AL} that 
\begin{equation}\label{eq:ALCheeger}
    \frac{(h^\sigma)^2}{2}\leq \lambda_1^\sigma\leq 2h^\sigma. 
\end{equation}
Multi-way signed Cheeger constants and the corresponding higher order Cheeger inequalities have been studied in \cite{AL,LLPP}.
\begin{remark}
For the particular case that $\sigma\equiv -1$, or, in general, the negation $(G,-\sigma)$ is balanced, we have \cite[Section 3.3]{AL}
\begin{equation}\label{eq:hbeta}
    h^\sigma_{|\sigma\equiv -1}=\beta.
\end{equation}
Moreover, we have 
\begin{equation}\label{lambdasigma_lambdaN}
    {\lambda^{\sigma}_1}_{|\sigma\equiv -1}=2-\lambda_N
\end{equation} since $\Delta^\sigma_{|\sigma\equiv -1}=2I-\Delta$. Therefore, the dual Cheeger inequality (\ref{eq:TBJ}) turns out to be a special case of (\ref{eq:ALCheeger}).
\end{remark}
\begin{remark}\label{rmk:curvature}
In \cite{LMP}, the Bakry-\'Emery $\Gamma$-calculus has been developed on signed graphs. For signed graphs with nonnegative curvature in the sense of curvature dimension inequality $CD^\sigma(0,\infty)$, it was shown for a $d$-regular signed graph that
\begin{equation*}
    \lambda_1^\sigma\leq 16(\log 2)d(h^\sigma)^2.
\end{equation*}
That is, we have $\lambda_1^\sigma\approx (h^\sigma)^2$ up to a factor depending only on $d$.
In the case that $\sigma\equiv -1$, and $CD^\sigma(0,\infty)_{|\sigma\equiv -1}$ holds, we have $2-\lambda_N\approx \beta^2$ up to a factor depending only on $d$. Unaware of the work \cite{LMP} (private communications with P. Ralli), it has been listed as an open question in \cite{MRT} asking  whether there is 
a definition of discrete curvature that permits the characterization of a class of
graphs for which $2-\lambda_N\approx \beta^2$. The work in \cite{LMP} is an extension of Klartag-Kosma-Ralli-Tetali \cite{KKRT}, which demonstrates that if a graph has nonnegative curvature in the sense of the classical Bakry-\'Emery curvature dimension inequality $CD(0,\infty)$, then $\lambda_2$ equals the square of the Cheeger constant $h$ up to a factor depending only on $d$. A particular class of graphs satisfying $CD(0,\infty)$ is abelian Cayley graphs.

Indeed, on a triangle-free graph $(G,\sigma)$ with $\sigma\equiv -1$, we have $CD^\sigma(0,\infty)_{|\sigma\equiv -1}$ holds if and only if $CD(0,\infty)$ holds \cite[Proposition 3.6]{LMP}. This is due to the switching invariance of $CD^\sigma(0,\infty)$. In particular, for triangle-free finite abelian Cayley graphs, we have both $CD(0,\infty)$ and $CD^\sigma(0,\infty)_{|\sigma\equiv -1}$, and hence $\lambda_2\approx h^2$ and $2-\lambda_N\approx \beta^2$. For methods of calculating the curvature dimension inequalities explicitly on signed graphs, see the recent work \cite{HL}.
\end{remark}

\section{New spectral and isoperimetric constants on signed graphs}\label{new_concept}
In this section, we introduce two new concepts on signed graphs. One is a signed version of the vertex isoperimetric constant. The other one is a signed version of the Poincar\'e-type functional constant \`{a} la Bobkov, Houdr\'e and Tetali. 

Let $(G,\sigma)$ be a signed graph and $\pi: V\to\mathbb{R}_{>0}$ be an associated vertex measure. 
Inspired by (\ref{eq:betaout}) and (\ref{eq:edgeiota}), we  introduce the $\infty$-frustration index on a subset $V_1\subseteq V$ as follows:
\begin{align*}
\iota^{\sigma,\pi}_{\infty}(V_{1})&=\min_{\tau:V_{1}\rightarrow \{+1,-1\}}\frac{1}{2}\sum_{x\in V_{1}}\sup_{y:y\in V_{1},y\sim x }|\tau(x)-\sigma_{xy}\tau(y)|\pi(x).
\end{align*}
It is direct to see that $\iota^{\sigma,\pi}_{\infty}(V_1)=0$ if and only if $\iota^\sigma(V_1)=0$. In fact, the $p$-frustration index with $p\in [1,\infty)$
\[\iota^\sigma_p(V_1)=\min_{\tau:V_{1}\rightarrow \{+1,-1\}}\frac{1}{2}\sum_{x\in V_{1}}\sum_{y:y\in V_{1},y\sim x}|\tau(x)-\sigma_{xy}\tau(y)|^p\]
has been studied in \cite{BGKLM} (see also \cite{BSS}), and we have $\iota^\sigma(V_1)=\iota^\sigma_p(V_1)$ for $p=1$.

Next, we introduce the concept of vertex isoperimetric constants for signed graphs.
\begin{definition}\label{def:beta}
For a signed graph $(G,\sigma)$ with a vertex measure $\pi$, we define the \emph{outer vertex isoperimetric constant} $h_{out}^\sigma$ as 
\begin{equation*}
    h^{\sigma}_{out}:=\min_{\emptyset\neq V_{1}\subseteq V}\frac{2\iota^{\sigma,\pi}_{\infty}(V_{1})+\pi(\partial_{out}(V_{1}))}{\pi(V_{1})},
\end{equation*}
and the \emph{symmetric vertex isoperimetic constant} as 
\begin{equation*}
    h_S^{\sigma}:=\min_{\emptyset\neq V_{1}\subseteq V}\frac{2\iota^{\sigma,\pi}_{\infty}(V_{1}^{\circ})+\pi(\partial_{S}(V_{1}))}{\pi(V_{1})},
\end{equation*}
where $V_{1}^{\circ}=V_1\setminus \partial_{in}(V_1)$ is the inner vertex subset of $V_{1}$.
\end{definition}
\begin{remark}\label{rmk:hsigmaouth}
    It is direct to check for a $d$-regular graph associated with a counting measure $\pi$, i.e., $\pi\equiv 1$, that
    \[\frac{1}{2d}h_{out}^{\sigma}\leq h^\sigma\leq h_{out}^\sigma.\]
\end{remark}
\begin{proposition}\label{prop:h=beta}
    On a signed graph $(G,\sigma)$ with $\sigma\equiv -1$ and $\pi$ being the counting measure, we have
    \begin{equation}\label{eq:houtbetaout}
        h_{out}^\sigma\geq \beta_{out}.
    \end{equation}
\end{proposition}
\begin{proof}
   In the case of $\sigma\equiv -1$ and $\pi$ being the counting measure, we check for any subset $V_1\subseteq V$ that
  \begin{align*}
      \iota^\sigma_{\infty}(V_1)=\min_{L\cap R=\emptyset, L\cup R=V_1}(I(L)+I(R))
  \end{align*}
  by observing the following correspondence between a switching function $\tau:V_1\to \{+1,-1\}$ and a partition $L,R$ of $V_1$: $L=\{x\in V_1:\tau(x)=+1\}$ and $R=\{x\in V_1: \tau(x)=-1\}$. 
  Therefore, we derive
\begin{align}
h^{\sigma}_{out}&\geq\min_{\emptyset\neq V_{1}\subseteq V}\frac{\iota^{\sigma,\pi}_{\infty}(V_{1})+\pi(\partial_{out}(V_{1}))}{\pi(V_{1})}\notag
\\&=\min_{\emptyset\neq V_1\subseteq V}\frac{\min_{L\cap R=\emptyset,L\cup R=V_1}(I(L)+I(R))+|\partial_{out}(V_1)|}{|V_1|}=\beta_{out}.\notag
\end{align}
\end{proof}

We define the following Poincar\'e-type functional constant \`{a} la Bobkov, Houdr\'e and Tetali on signed graphs. 
\begin{definition}\label{def:BHT}
For a signed graph $(G,\sigma)$ with an associated vertex measure $\pi$, we define $\lambda_{\infty}^{\sigma}$ as the optimal constant in the following Poincar\'e-type inequality 
\begin{equation}\label{eq:BHTsigma}\lambda_{\infty}^{\sigma}\sum_{x\in V}f(x)^{2}\pi(x)\leq\sum_{x\in V}\sup_{y:y\sim x}|f(x)-\sigma_{xy}f(y)|^{2}\pi(x).\end{equation}
where $f:V\rightarrow \mathbb{R}$ is arbitrary.
\end{definition}
It is direct to check that $\lambda_{\infty}^\sigma$ is switching invariant.
In the case of $\sigma\equiv +1$, or, in general, $(G,\sigma)$ is balanced, the constant $\lambda^{\sigma}_{\infty}$ is exactly the constant $\lambda_{\infty}$ introduced by Bobkov, Houdr\'e and Tetali in \cite[Section 2]{BHT}.
\begin{lemma}\label{lem:lambdainfty} For a $d$-regular signed graph $(G,\sigma)$ with the vertex measure $\pi$ being the counting measure, we have 
\[2\lambda_1^\sigma\leq \lambda_{\infty}^\sigma\leq 2d\lambda_1^\sigma.\]
\end{lemma}
\begin{proof}
By definition, we have 
$$\lambda_{\infty}^{\sigma}=\inf_{\substack{f:V\rightarrow\mathbb{R}\\ f\not\equiv 0}}\frac{\sum_{x\in V}\sup_{y:y\sim x}|f(x)-\sigma_{xy}f(y)|^{2}\pi(x)}{\sum_{x\in V}f(x)^{2}\pi(x)}.$$
On the other hand, for a $d$-regular signed graph $(G,\sigma)$, we have
\begin{align*}
\lambda_1^\sigma=\inf_{\substack{f:V\rightarrow\mathbb{R}\\ f\not\equiv 0}}\frac{\sum_{x\in V}f(x)\Delta^\sigma f(x)d_x}{\sum_{x\in V}f(x)^2d_x}=\inf_{\substack{f:V\rightarrow\mathbb{R}\\ f\not\equiv 0}}\frac{\frac{1}{2}\sum_{x\in V}\sum_{y:y\sim x}(f(x)-\sigma_{xy}f(y))^2}{d\sum_{x\in V}(f(x))^2}.
\end{align*}
When $\pi$ is the counting measure, $2\lambda_1^\sigma\leq\lambda_{\infty}^\sigma\leq 2d\lambda_1^\sigma$ follows directly.
\end{proof}

\section{Cheeger-type inequalities on signed graphs}\label{signed_graph}
In this section, we present our main estimates relating the constants  $\lambda^{\sigma}_{\infty}$, $h^{\sigma}_{out}$ and $h_S^{\sigma}$ of signed graphs.
\begin{theorem}\label{thm:main}
For a signed graph $(G,\sigma)$ with a vertex measure $\pi$, we have 
\begin{align}\label{ineq:beta}
\lambda^{\sigma}_{\infty}\geq \left(\sqrt{1+h_{out}^\sigma}-1\right)^2,
\end{align}
and
\begin{align}\label{ineq:eta}
\lambda^{\sigma}_{\infty}\geq \left(\sqrt{1+h_{S}^\sigma}-1\right)^2.
\end{align}
\end{theorem}
For the proof, the following lemma taken from \cite[Lemma 4.1]{AL} (see in particular \cite[(4.2)]{AL}) plays an important role.
\begin{lemma}\label{lemma:ineq c}
For any $t\in [0,1]$, we denote by $Y_{t}:[-1,1]\rightarrow\{+1,-1,0\}$ the indicator function such that for any $z\in [-1,1]$ 
$$Y_{t}(z):=\left\{
\begin{aligned}
\frac{z}{|z|},\textrm{ if }|z|\geq t,\\
0,\textrm{ if }|z|< t,
\end{aligned}
\right.$$
where we use the notation that $z/|z|=1$ when $z=0$. Then, we have for any $z_1,z_2\in [-1,1]$ the following inequality
\[\int_{0}^{1}\left|Y_{\sqrt{t}}(z_{1})-Y_{\sqrt{t}}(z_{2})\right|dt\leq |z_{1}-z_{2}|(|z_{1}|+|z_{2}|).\]
\end{lemma}

\begin{proof}[Proof of Theorem \ref{thm:main}]

For the derivation of $\lambda_{\infty}^{\sigma}$, it suffices to verify that (\ref{eq:BHTsigma}) for all functions $f:V\rightarrow\mathbb{R}$ satisfying $\max_{x\in V}|f(x)|=1$.


Let $f$ be such a function. For any given $t\in [0,1]$, we define the subset $V^f(\sqrt{t})$ of $V$ with respect to $f$ and $t$ as follows:
$$V^{f}(\sqrt{t})=\left\{x\in V\left|\,|f(x)|\geq\sqrt{t}\right.\right\}.$$
Let $Y_t:[-1,1]\to \{+1,-1,0\}$ be defined as in Lemma \ref{lemma:ineq c}. When restricting to the subset $V^f(\sqrt{t})$, the function $Y_{\sqrt{t}}(f)$ can be considered as a switching function.

Next, we consider the following crucial quantity
\begin{equation}\label{eq:keyquantity}
    \frac{\int_0^1\sum_{x\in V}\sup_{y:y\sim x}\left|Y_{\sqrt{t}}(f(x))-\sigma_{xy}Y_{\sqrt{t}}(f(y))\right|\pi(x)dt}{\int_0^1\pi\left(V^f(\sqrt{t})\right)dt}.
\end{equation}
For the denominator of the ratio above, we compute
\begin{align}\label{eq:denominator}
\int_{0}^{1}\pi\left(V^{f}(\sqrt{t})\right)dt=&\int_{0}^{1}\sum_{x\in V^f(\sqrt{t})}\pi(x)dt=\sum_{x\in V}\int_{0}^{f(x)^2}1\,dt\pi(x)=\sum_{x\in V}f(x)^2\pi(x).
\end{align}
For the numerator of the ratio, we derive by Lemma \ref{lemma:ineq c} that 
\begin{align}\label{ineq:sup}
&\int_{0}^{1} \sum_{x\in V} \sup_{y:y\sim x}\left|Y_{\sqrt{t}}(f(x))-\sigma_{xy}Y_{\sqrt{t}}(f(y))\right|\pi(x)dt\nonumber\\
=&\int_{0}^{1} \sum_{x\in V} \sup_{y:y\sim x}\left|Y_{\sqrt{t}}(f(x))-Y_{\sqrt{t}}(\sigma_{xy}f(y))\right|\pi(x)dt\nonumber\\
\leq&\sum_{x\in V}\sup_{y:y\sim x}|f(x)-\sigma_{xy}f(y)|(|f(x)|+|f(y)|)\pi(x).
\end{align}

On the other hand, we estimate the numerator of the ratio from below as follows.
For each edge $x\sim y$, the term $Y_{xy}:=\left|Y_{\sqrt{t}}(f(x))-\sigma_{xy}Y_{\sqrt{t}}(f(y))\right|$ can take the following values
\begin{itemize}
  \item $x\in V^{f}(\sqrt{t})$ and $y\in V^{f}(\sqrt{t})$, then $Y_{xy}=2$ or $0$;
  \item $x\in V^{f}(\sqrt{t})$ and $y\in V\setminus V^{f}(\sqrt{t})$, then $Y_{xy}=1$;
  \item $x\in V\setminus V^{f}(\sqrt{t})$ and $y\in V^{f}(\sqrt{t})$, then $Y_{xy}=1$;
  \item $x\in V\setminus V^{f}(\sqrt{t})$ and $y\in V\setminus V^{f}(\sqrt{t})$, then $Y_{xy}=0$.
\end{itemize}
Then we estimate 
\begin{align}
&\sum_{x\in V}\sup_{y:y\sim x}\left|Y_{\sqrt{t}}(f(x))-\sigma_{xy}Y_{\sqrt{t}}(f(y))\right|\pi(x)\notag\\
=&\sum_{x\in V^f(\sqrt{t})}\sup_{y:y\sim x}\left|Y_{\sqrt{t}}(f(x))-\sigma_{xy}Y_{\sqrt{t}}(f(y))\right|\pi(x)\notag\\
&+\sum_{x\in V\setminus V^f(\sqrt{t})}\sup_{y:y\sim x}\left|Y_{\sqrt{t}}(f(x))-\sigma_{xy}Y_{\sqrt{t}}(f(y))\right|\pi(x)\notag\\
\geq&\sum_{x\in V^f(\sqrt{t})}\sup_{\substack{y:y\sim x\\y\in V^f(\sqrt{t})}}\left|Y_{\sqrt{t}}(f(x))-\sigma_{xy}Y_{\sqrt{t}}(f(y))\right|\pi(x)+\pi\left(\partial_{out}(V^f(\sqrt{t}))\right)\notag\\
\geq &2\iota^{\sigma,\pi}_{\infty}\left(V^{f}(\sqrt{t})\right)+\pi\left(\partial_{out}(V^{f}(\sqrt{t}))\right).\label{eq:lowerboundhout}
\end{align}

Alternatively, we estimate as follows

\begin{align}
&\sum_{x\in V}\sup_{y:y\sim x}\left|Y_{\sqrt{t}}(f(x))-\sigma_{xy}Y_{\sqrt{t}}(f(y))\right|\pi(x)\notag\\
=&\sum_{x\in V^f(\sqrt{t})^\circ}\sup_{y:y\sim x}\left|Y_{\sqrt{t}}(f(x))-\sigma_{xy}Y_{\sqrt{t}}(f(y))\right|\pi(x)\notag\\
&+\sum_{x\in V\setminus V^f(\sqrt{t})^\circ}\sup_{y:y\sim x}\left|Y_{\sqrt{t}}(f(x))-\sigma_{xy}Y_{\sqrt{t}}(f(y))\right|\pi(x)\notag\\
\geq&\sum_{x\in V^f(\sqrt{t})^\circ}\sup_{\substack{y:y\sim x\\y\in V^f(\sqrt{t})^\circ}}\left|Y_{\sqrt{t}}(f(x))-\sigma_{xy}Y_{\sqrt{t}}(f(y))\right|\pi(x)+\pi\left(\partial_{S}(V^f(\sqrt{t}))\right)\notag\\
\geq &2\iota^{\sigma,\pi}_{\infty}\left(V^{f}(\sqrt{t})^\circ\right)+\pi\left(\partial_{S}(V^{f}(\sqrt{t}))\right).\label{eq:lowerboundhS}
\end{align}

Combining (\ref{eq:denominator}), (\ref{ineq:sup}) and (\ref{eq:lowerboundhout}), we arrive at 
\begin{align}
&\frac{\int_{0}^{1}2\iota^{\sigma,\pi}_{\infty}\left(V^{f}(\sqrt{t})\right)+\pi\left(\partial_{out}(V^{f}(\sqrt{t}))\right)dt}{\int_{0}^{1}\pi\left(V^{f}(\sqrt{t})\right)dt}\notag\\
\leq& \frac{\sum_{x\in V}\sup_{y:y\sim x}|f(x)-\sigma_{xy}f(y)|(|f(x)|+|f(y)|)\pi(x)}{\sum_{x\in V}f(x)^2\pi(x)}.\label{eq:step1}
\end{align}

For each summand in the numerator of the ratio above, we estimate
\begin{align*}
&\sup_{y:y\sim x}|f(x)-\sigma_{xy}f(y)|(|f(x)|+|f(y)|)\pi(x)\\
=&\sup_{y:y\sim x}|f(x)-\sigma_{xy}f(y)|(|f(y)|-|f(x)|)\pi(x)+\sup_{y:y\sim x}|f(x)-\sigma_{xy}f(y)|(2|f(x)|)\pi(x)\\
\leq&\sup_{y:y\sim x}|f(x)-\sigma_{xy}f(y)|^{2}\pi(x)+2\sup_{y:y\sim x}|f(x)-\sigma_{xy}f(y)||f(x)|\pi(x).
\end{align*}
Then, we obtain by Cauchy-Schwarz inequality that
\begin{align}\label{ineq:Cauchy-Schwarts}
&\sum_{x\in V}\sup_{y:y\sim x}|f(x)-\sigma_{xy}f(y)|(|f(x)|+|f(y)|)\pi(x)\nonumber\\
\leq&\sum_{x\in V}\sup_{y:y\sim x}|f(x)-\sigma_{xy}f(y)|^{2}\pi(x)+2\sum_{x\in V}\sup_{y:y\sim x}|f(x)-\sigma_{xy}f(y)||f(x)|\pi(x)\nonumber\\
\leq&\sum_{x\in V}\sup_{y:y\sim x}|f(x)-\sigma_{xy}f(y)|^{2}\pi(x)\nonumber\\
&\hspace{2cm}+2\left(\sum_{x\in V}\sup_{y:y\sim x}|f(x)-\sigma_{xy}f(y)|^{2}\pi(x)\right)^{\frac{1}{2}}\left(\sum_{x\in V}f(x)^{2}\pi(x)\right)^{\frac{1}{2}}.
\end{align}
For simplicity, we denote by
$$A^2:=\sum_{x\in V}\sup_{y:y\sim x}|f(x)-\sigma_{xy}f(y)|^{2}\pi(x),\,\,\text{and}\,\,B^2:=\sum_{x\in V}f(x)^{2}\pi(x),$$
Inserting (\ref{ineq:Cauchy-Schwarts}) into (\ref{eq:step1}) gives
\begin{align*}
\frac{\int_{0}^{1}2\iota^{\sigma,\pi}_{\infty}\left(V^{f}(\sqrt{t})\right)+\pi\left(\partial_{out}(V^{f}(\sqrt{t}))\right)dt}{\int_{0}^{1}\pi\left(V^{f}(\sqrt{t})\right)dt}\leq \frac{A^{2}+2AB}{B^{2}}.
\end{align*}
Notice that there exists $t'\in [0,1]$ such that 
\begin{align*}
\frac{\int_{0}^{1}2\iota^{\sigma,\pi}_{\infty}\left(V^{f}(\sqrt{t})\right)+\pi\left(\partial_{out}(V^{f}(\sqrt{t}))\right)dt}{\int_{0}^{1}\pi\left(V^{f}(\sqrt{t})\right)dt}\geq \frac{2\iota^{\sigma,\pi}_{\infty}\left(V^{f}(\sqrt{t'})\right)+\pi\left(\partial_{out}(V^{f}(\sqrt{t'}))\right)}{\pi\left(V^{f}(\sqrt{t'})\right)}.
\end{align*}
Therefore, we have 
$$\frac{A^{2}+2AB}{B^{2}}\geq h^{\sigma}_{out}.$$
Solving it directly yields
$$\frac{A^{2}}{B^{2}}\geq\left(\sqrt{1+h^{\sigma}_{out}}-1\right)^{2}.$$
By definition we derive
$$\lambda^{\sigma}_{\infty}\geq \left(\sqrt{1+h^{\sigma}_{out}}-1\right)^{2}.$$

Along the same line, we derive by combining (\ref{eq:denominator}), (\ref{ineq:sup}) and (\ref{eq:lowerboundhS}) (instead of (\ref{eq:lowerboundhout})) that
$$\lambda^{\sigma}_{\infty}\geq \left(\sqrt{1+h^{\sigma}_{S}}-1\right)^{2}.$$
This completes the proof.
\end{proof}
\begin{remark}
    Observe that 
    \[h_{out}^\sigma\leq \frac{2\iota^{\sigma,\pi}_{\infty}(V)}{\pi(V)}\leq 2\,\,\text{and}\,\,h_{S}^\sigma\leq \frac{2\iota^{\sigma,\pi}_{\infty}(V)}{\pi(V)}\leq 2\]
    are both uniformly bounded. Hence our estimate tells that
    \[\lambda_{\infty}^\sigma\geq \frac{1}{12}(h_{out}^\sigma)^2\,\,\text{and}\,\,\lambda_{\infty}^\sigma\geq \frac{1}{12}(h_{S}^\sigma)^2.\]
\end{remark}
\begin{proof}[Proof of Theorem \ref{thm: Main2}]
The inequality in (\ref{eq:key}) follows from combining the inequality (\ref{ineq:beta})  and Lemma \ref{lem:lambdainfty}.
\end{proof}

\begin{remark}\label{rem:buser}
    Combining Lemma \ref{lem:lambdainfty}, (\ref{eq:ALCheeger}) and Remark \ref{rmk:hsigmaouth}, we have for a $d$-regular signed graph that
    \[\lambda^\sigma_{\infty}\leq 2d\lambda_1^\sigma\leq 4dh^\sigma \leq 4dh^\sigma_{out}.\]
    Recall from Remark \ref{rmk:curvature}, we have for a $d$-regular signed graphs with $CD^\sigma(0,\infty)$ \cite{LMP} that 
\[\lambda^\sigma_{\infty}\leq 2d\lambda_1^\sigma\leq 32(\log 2)(dh^\sigma)^2 \leq 32(\log 2)(dh_{out}^\sigma)^2.\]
In particular, Theorem \ref{thm: Main2} is tight up to a factor depending only on $d$. We wonder whether the order of $d$ in the estimate, $\lambda_1^\sigma\geq c(h_{out}^\sigma)^2/d$ for some constant $c$, is optimal or not.
\end{remark}
\section{Applications on non-bipartite finite Cayley graphs, Cayley sum graphs and vertex transitive graphs}\label{Cayley_graph}
Besides Theorem \ref{thm:main}, another key ingredient to prove  Theorem \ref{thm:Main1} is the following estimate of Moorman, Ralli and Tetali.
\begin{theorem}(Moorman, Ralli and Tetali \cite[Theorem 2.4]{MRT})\label{thm:MRThoutbetaout}
Let  $C(X,S)$ be a non-bipartite Calyley graph of a finite group $X$ with respect to a generating set $S$ satisfying $S^{-1}=S$ and $id_{X}\notin S$. Let $d=|S|$. Then we have
$$h_{out}\leq200\beta_{out}.$$
\end{theorem}
If we allow self-loops in a Cayley graph, that is, the set $S$ contains the identity element, Theorem \ref{thm:MRThoutbetaout} holds trivially true since $\beta_{out}=1$ and $h_{out}\leq 2$. If a Cayley graph is disconnected, that is, the set $S$ is not a generating one, Theorem \ref{thm:MRThoutbetaout} also holds trivially true since $h_{out}=0\leq \beta_{out}$.

We are now prepared to show Theorem \ref{thm:Main1}.
\begin{proof}[Proof of Theorem \ref{thm:Main1}]
It remains to show the lower bound of the interval. We apply Theorem \ref{thm: Main2} in the particular case that $\sigma\equiv -1$: Since $\Delta^\sigma_{|\sigma\equiv -1}=2I-\Delta$, 
we have 
\[{\lambda_1^\sigma}_{|\sigma\equiv -1}=2-\lambda_N=1+t_1.\]
Moreover, we have by Proposition \ref{prop:h=beta} that 
    \[{h^{\sigma}_{out}}_{|\sigma\equiv -1}\geq \beta_{out}.\]
Therefore, applying (\ref{eq:key}) in the case that $\sigma\equiv -1$ yields for a $d$-regular graph $G$ that
\begin{equation}\label{eq:key1}
    2-\lambda_N=1+t_1\geq \frac{\left(\sqrt{1+\beta_{out}}-1\right)^2}{2d}.
\end{equation}
Since $\beta_{out}\leq 1$ (see (\ref{eq:betaoutbound})), we have
\begin{equation}\label{eq:gapbetaout}
    2-\lambda_N\geq \frac{\beta_{out}^2}{16d}.
\end{equation}
Combining (\ref{eq:gapbetaout}) and Theorem \ref{thm:MRThoutbetaout} completes the proof.
\end{proof}

From the proof of Theorem \ref{thm:Main1}, we see that, for any non-bipartite $d$-regular graph satisfying $h_{out}\leq C_0\beta_{out}$ for some absolute constant $C_0$, the nontrivial eigenvalues of its normalized adjacency matrix lie in an interval 
\begin{equation}\label{eq:section5key}
      \left[ -1+\frac{ch_{out}^2}{d},1-\frac{Ch_{out}^2}{d}\right], 
\end{equation}
 for some absolute constants $c$ and $C$. In the remaining part of this section, we prove Theorem \ref{thm:Main3} saying that (\ref{eq:section5key}) holds for non-bipartite vertex transitive graphs and  non-bipartite Cayley sum graphs via extending Theorem \ref{thm:MRThoutbetaout}.

We first deal with Cayley sum graphs.

\begin{theorem}\label{thm:cayleysumgraph}
    Let $G=C_{\Sigma}(X,S)$ be a non-bipartite Cayley sum graph of a finite group $X$ with respect to a normal subset $S$.Then there exists an absolute constant $C$ such that
$$h_{out}\leq C\beta_{out}.$$
\end{theorem}
Our proof follows closely the proof of Moorman-Ralli-Tetali \cite[Theorem 2.4]{MRT}. One key modification is to use the following fact: For any $g\in X$, the map $\phi_g: X\to X$ defined via $\phi_g(x):=gxg^{-1}$ for each $x\in X$, provides an automorphism of the Cayley sum graph $C_{\Sigma}(X,S)$.

We prepare two lemmas from set theory for the proof, the first of which is straightforward. 
\begin{lemma}\label{lemma:subsetabc}
    For any subsets $A$, $B$ and $C$ of $X$, it holds that
    $$|A|+|B|+|C|-|X|\leq |A\cap B|+|B\cap C|+|C\cap A|.$$
\end{lemma}

\begin{lemma}\label{lemma:subsetabcd}
    Let $G=(V,E)$ be a graph. Let $A$, $B$, $C$ and $D$ be subsets of $V$, then we have the following inequality:
    $$|\partial_{out}[(A\cap C)\cup(B\cap D)]|\leq I(A)+I(B)+I(C)+I(D)+|\partial_{out}(A\cup B)|+|\partial_{out}(C\cup D)|.$$
\end{lemma}
\begin{proof}
    Let $x\in\partial_{out}[(A\cap C)\cup(B\cap D)]$. We consider the following two cases respectively.

    \textbf{Case $1$}. $x\in(B\cap C)\cup(A\cap D)$. Then we have $\left(x\in(B\cap C)\,\text{ or }\,x\in(A\cap D)\right)$  and $\left(\exists \,y\in (A\cap C),\,\text{ such that}\, y\sim x\,\text{ or }\,\exists\, y\in(B\cap D),\,\text{ such that}\, y\sim x\right)$. This implies $4$ possibilities:
    \begin{itemize}
        \item $x\in(B\cap C)$, $\exists\, y\in (A\cap C)$ such that $y\sim x$, then it can be counted in $I(C)$.
        \item $x\in(B\cap C)$, $\exists\, y\in (B\cap D)$ such that $y\sim x$, then it can be counted in $I(B)$.
        \item $x\in(A\cap D)$, $\exists\, y\in (A\cap C)$ such that $y\sim x$, then it can be counted in $I(A)$.
        \item $x\in(A\cap D)$, $\exists\, y\in (B\cap D)$ such that $y\sim x$, then it can be counted in $I(D)$.
    \end{itemize}
    Hence, we obtain
    $$|\partial_{out}[(A\cap C)\cup(B\cap D)]\cap[(B\cap C)\cup(A\cap D)]|\leq I(A)+I(B)+I(C)+I(D).$$
    \textbf{Case $2$}. $x\notin (B\cap C)\cup(A\cap D)$. Then $x\notin (B\cap C)$ and $x\notin(A\cap D)$ and $x\notin(A\cap C)$ and $x\notin(B\cap D)$, and $\left(\exists\, y\in (A\cap C)\,\text{ such that }\,y\sim x\,\text{ or }\,\exists\, y\in(B\cap D)\,\text{ such that }\,y\sim x\right)$. This implies the following possibilities:
  \begin{itemize}
      \item If $x\notin(A\cup B)$, then $x\in \partial_{out}(A\cup B)$.
      \item IF $x \in A$, then $x\in A\setminus (C\cup D)$. Hence $x\in\partial_{out}(C\cup D)$.
      \item If $x\in B$, then $x\in B\setminus (C\cup D)$. Hence $x\in \partial_{out}(C\cup D)$.
  \end{itemize}
Therefore, we derive
  $$\partial_{out}[(A\cap C)\cup (B\cap D)]\setminus [(B\cap C)\cup(A\cap D)]\subseteq\partial_{out}(A\cup B)\cup\partial_{out}(C\cup D).$$

Collecting facts obtained in Case $1$ and Case $2$ completes the proof.
\end{proof}
Next, we prove Theorem \ref{thm:cayleysumgraph}.
\begin{proof}[Proof of Theorem \ref{thm:cayleysumgraph}]
Let $\epsilon>0$ be a constant to be determined later. We first observe that if $\beta_{out}\geq \epsilon$, then we have 
\[\beta_{out}\geq \epsilon\geq \frac{\epsilon}{2}h_{out},\]
since $h_{out}\leq 2$ is universally bounded. 

From now on, we assume $\beta_{out}<\epsilon$.
Suppose $L$ and $R$ be the two subsets achieving $\beta_{out}$, that is,
$$\beta_{out}=\frac{I(L)+I(R)+|\partial_{out}(L\cup R)|}{|L\cup R|}.$$
Consequently, we have $|\partial_{out}(L\cup R)|\leq \beta_{out}|L\cup R|$. We define 
\[Y=\{x\in X\,|\,\mathrm{dist}_{G}(x,L\cup R)\geq 2\}.\] 
Hence, we have $\partial_{out}(Y)=\partial_{out}(L\cup R)$ if $Y\neq\emptyset$. We consider the following three cases. Denote $N=|X|$.
    
\textbf{Case $1$}. $\epsilon N<|Y|<\frac{1}{2}N$. Then by definition of $h_{out}$, we derive
$$h_{out}\leq\frac{|\partial_{out}(Y)|}{|Y|}=\frac{|\partial_{out}(L\cup R)|}{|Y|}\leq\frac{\beta_{out}|L\cup R|}{\epsilon N}<\frac{1}{\epsilon}\beta_{out}.$$

\textbf{Case $2$}. $|Y|\geq\frac{1}{2}N$. Then $|L\cup R|<\frac{1}{2}N$ since $Y$ is disjoint with $L\cup R$. And we arrive at
$$h_{out}\leq\frac{|\partial_{out}(L\cup R)|}{|L\cup R|}\leq\beta_{out}.$$

\textbf{Case $3$}. $|Y|\leq\epsilon N$. For any element $g\in X$, we define the following two subsets:
$$A(g)=(gLg^{-1}\cap L)\cup(gRg^{-1}\cap R),\,\,\text{and}\,\,B(g)=(gRg^{-1}\cap L)\cup(gLg^{-1}\cap R).$$
First observe that $A(g)$ and $B(g)$ are disjoint. Indeed, we have 
\begin{align*}
A(g)\cap B(g)=&[(gLg^{-1}\cap L)\cap(gLg^{-1}\cap R)]\cup[(gLg^{-1}\cap L)\cap(gRg^{-1}\cap L)]\\
\cup&[(gRg^{-1}\cap R)\cap(gLg^{-1}\cap R)]\cup[(gRg^{-1}\cap R)\cap(gRg^{-1}\cap L)],
\end{align*}
where the $4$ subsets in square brackets are all empty since $L$ and $R$ are disjoint.
By Lemma \ref{lemma:subsetabcd}, we estimate 
$$|\partial_{out}(A(g))|\leq I(L)+I(R)+I(gLg^{-1})+I(gRg^{-1})+|\partial_{out}(L\cup R)|+|\partial_{out}(g(L\cup R)g^{-1})|,$$
and
$$|\partial_{out}(B(g))|\leq I(L)+I(R)+I(gLg^{-1})+I(gRg^{-1})+|\partial_{out}(L\cup R)|+|\partial_{out}(g(L\cup R)g^{-1})|.$$

We claim that 
\begin{equation}\label{eq:claim}
    I(L)=I(gLg^{-1}),\,\, I(R)=I(gRg^{-1})\,\,\text{and}\,\, |\partial_{out}(L\cup R)|=|\partial_{out}(g(L\cup R)g^{-1})|.
\end{equation}
Let us define the map $\phi_g: X\to X$ such that $\phi_g(x)=gxg^{-1}$ for any $x\in X$. Then we observe that
\[\phi_g(x^{-1}s)=\phi_g(x)^{-1}\phi_g(s),\,\,\text{for any}\,\,x\in X, \,s\in S.\]

Applying the claim (\ref{eq:claim}), we arrive at
\begin{equation*}
    |\partial_{out}(A(g))|,\, |\partial_{out}(B(g))|\leq 2\left(I(L)+I(R)+|\partial_{out}(L\cup R)|\right).
\end{equation*}
Since $A(g)$ and $B(g)$ are disjoint, at least one of them has cardinality no greater than $N/2$. Hence, we estimate
\begin{equation}\label{eq:houtAgBg}
    h_{out}\leq \frac{2\left(I(L)+I(R)+|\partial_{out}(L\cup R)|\right)}{\min\{|A(g)|,|B(g)|\}}\leq \frac{2\beta_{out}|L\cup R|}{\min\{|A(g)|,|B(g)|\}}.
\end{equation}
Next, we aim at estimating the denominator. 

Without loss of generality, we assume $|L|\geq |R|$. 
Then, we bound $|L|$ from below
\begin{equation*}
    |L|\geq \frac{|L\cup R|}{2}=\frac{N-|Y|-|\partial_{out}(L\cup R)|}{2}\geq \frac{1-\epsilon-\beta_{out}}{2}N.
\end{equation*}
Suppose that $|L|\geq \frac{1+\epsilon}{2}N$. Then we derive that $|L|-I(L)\leq |X\setminus L|$ by considering the number of edges connection the sets $\{x\in X\,|\,\text{there is no neighboring vertex of $x$ in $L$}\}$ and $X\setminus L$. This tells $I(L)\geq \epsilon N$ and, hence, $\beta_{out}\geq \frac{I(L)}{N}\geq \epsilon$.

Therefore, it remains to consider the case 
\begin{equation}\label{eq:Lbounds}
     \frac{1-\epsilon-\beta_{out}}{2}N\leq |L|\leq \frac{1+\epsilon}{2}N.
\end{equation}
In this case, we claim there exists an element $g\in X$, such that 
\begin{equation}\label{eq:claimdelta}
    |L\cap gLg^{-1}|\in (\delta|L|, (1-\delta)|L|)
\end{equation} for some $\delta>0$ to be determined. We prove this claim by contradiction. Suppose that there is no such $g\in X$. We will show that the graph is bipartite. 

Define two sets
\[X_1=\{g\in X\,|\,|gLg^{-1}\cap L|\geq (1-\delta)|L|\}\,\,\text{and}\,\,X_2=\{g\in X\,|\,|gLg^{-1}\cap L|\leq \delta|L|\}.\]
By assumption, this gives a partition of $X$.
First, we have $X_1^2\subseteq X_1$ if $\delta<\frac{1}{3}$. Indeed, we have for any $g,h\in X_1$ that
\begin{align*}
    |(gh)L(gh)^{-1}\cap L|=&|L|-|L\setminus (gh)L(gh)^{-1}|\geq |L|-|L\setminus gLg^{-1}|-|gLg^{-1}\setminus (gh)L(gh)^{-1}|\\
    =&|L\cap gLg^{-1}|-|L\setminus hLh^{-1}|=|L\cap gLg^{-1}|-|L|+|L\cap hLh^{-1}|\\
    \geq &(1-2\delta)|L|>\delta |L|,
\end{align*}
where we use $\delta<\frac{1}{3}$ and the inequality $|A\setminus B|+|B\setminus C|\geq |A\setminus C|$ for any subsets $A,B$ and $C$. This tells that $X_1$ is a subgroup of $X$.

Secondly, we have $X_2^2\subseteq X_1$ if $\delta<1-\frac{2}{3(1-\epsilon-\beta_{out})}$. Indeed, we have for any $g,h\in X_2$ that
\begin{align*}
    &|(gh)L(gh)^{-1}\cap L|\\
    \geq &-|L\cap gLg^{-1}|-|gLg^{-1}\cap (gh)L(gh)^{-1}|+|L|+|gLg^{-1}|+|(gh)L(gh)^{-1}|-|X|\\
    =&-|L\cap gLg^{-1}|-|L\cap hLh^{-1}|+3|L|-N\\
    \geq & -2\delta|L|+3|L|-\frac{2}{1-\epsilon-\beta_{out}}|L|
    >\delta |L|,
\end{align*}
where we use Lemma \ref{lemma:subsetabc}, (\ref{eq:Lbounds}) and $\delta<1-\frac{2}{3(1-\epsilon-\beta_{out})}$.

We observe that
\begin{align*}
    \sum_{g\in X_1}|L\cap gLg^{-1}|=&\sum_{g\in X_1}\sum_{x\in L: \,\exists\,y\in L\,\text{such that}\,gyg^{-1}=x}1=\sum_{x\in L}\sum_{g\in X_1: g^{-1}xg\in L}1\\
    =&\sum_{x\in L\cap X_1}\sum_{g\in X_1: g^{-1}xg\in L}1+\sum_{x\in L\cap X_2}\sum_{g\in X_1: g^{-1}xg\in L}1\\
    =&\sum_{x\in L\cap X_1}\sum_{g\in X_1: g^{-1}xg\in L\cap X_1}1+\sum_{x\in L\cap X_2}\sum_{g\in X_1: g^{-1}xg\in L\cap X_2}1\\
    \leq & |L\cap X_1|^2+|L\cap X_2|^2=|L|^2-2|L\cap X_1||L\cap X_2|.
\end{align*}
This leads to
\begin{equation}\label{eq:LcapX1LcapX2}
    (1-\delta)|L||X_1|\leq |L|^2-2|L\cap X_1||L\cap X_2|.
\end{equation}
Assume that $X_2$ is empty. Then, we have $|X_1|=N$. So, if $\delta<\frac{1}{3}$, we estimate 
\[|L|\geq (1-\delta)N>\frac{2}{3}N.\]
This contradicts to $|L|\leq\frac{1+\epsilon}{2}N$ in (\ref{eq:Lbounds}) once we choose $\epsilon<\frac{1}{3}$.

Therefore, the subset $X_2$ is not empty. Then $X_1$ is a proper subgroup of $X$ with its complement $X_2$ satisfying $X_2^2\subseteq X_1$. Hence, $X_1$ is a subgroup of index $2$ and $X_2$ is its unique nontrivial coset. In particular, we have $|X_1|=|X_2|=\frac{N}{2}$. 
Then we derive from (\ref{eq:LcapX1LcapX2}) that
\begin{align*}
    \frac{1-\delta}{2}N|L|\leq \frac{1+\epsilon}{2}N|L|-2|L\cap X_1||L\cap X_2|.
\end{align*}
This yields that
\begin{equation*}
    |L\cap X_1||L\cap X_2|\leq \frac{\epsilon+\delta}{4}N|L|.
\end{equation*}
Consequently, there exists an $i\in \{1,2\}$ such that $|L\cap X_i|\leq \frac{1}{2}\sqrt{(\epsilon+\delta)N|L|}$. Let $j\in\{1,2\}\setminus\{i\}$ be the other index, for which we have 
\begin{equation}\label{ineq:XiXj}
    |L\cap X_j|=|L|-|L\cap X_i|\geq |L|-\frac{1}{2}\sqrt{(\epsilon+\delta)N|L|}.
\end{equation}

Let us denote by $A^c$ the complement of a subset $A$. For any $s\in S$, we estimate
\begin{align*}
    |X_j^{-1}s\cap X_j|= & |(X_j\cap L)^{-1}s\cap X_j\cap L|+|(X_j\cap L)^{-1}s\cap X_j\cap L^c|+|(X_j\cap L^c)^{-1}s\cap X_j|\\
    \leq & I(L)+|X_j\cap L^c|+|(X_j\cap L^c)^{-1}s|=I(L)+2|X_j\cap L^c|\\
    \leq &\beta_{out}N+2(|X_j|-|X_j\cap L|)
    \leq  \beta_{out}N+N-2|L|+\sqrt{(\epsilon+\delta)N|L|}\\
    \leq & \left(\epsilon+2\beta_{out}+\sqrt{\frac{(\epsilon+\delta)(1+\epsilon)}{2}}\right)N.
\end{align*}
The second inequality above is from the definition of $\beta_{out}$ and the fact that $L$ and $R$ separate the vertex set $V$ such that $|L\cup R|=N$. The third inequality is from the inequality (\ref{ineq:XiXj}) and the fact that $X_1$ is a subgroup of index $2$ with its unique nontrivial coset $X_2$. 
Once we choose $\delta<\left(\frac{1}{2}-\epsilon-2\beta_{out}\right)^2\frac{2}{1+\epsilon}-\epsilon$, we have $|X_j^{-1}s\cap X_j|<\frac{N}{2}$. This implies that \[S\cap X_1=\emptyset,\] since if there exists $s\in X_1\cap S$, we have $|X_j^{-1}s\cap X_j|=|X_j|=\frac{N}{2}$, which is a contradiction.

Then we derive that the graph is bipartite. In fact, if $g\in X_1$ and $s\in S$, then $g^{-1}s\in g^{-1}X_2=X_2$. Similarly, if $g\in X_2$ and $s\in S$, then $g^{-1}s\in g^{-1}X_2=X_1$. Therefore, the graph is bipartite, which is a contradiction. The proves our claim that there exists an element $g\in X$ such that
(\ref{eq:claimdelta}) holds.

Now we can carry out the estimates of $|A(g)|$ and $|B(g)|$. Let $g$ be the one satisfying (\ref{eq:claimdelta}). Then we derive
\begin{equation*}
    |A(g)|\geq |gLg^{-1}\cap L|\geq \delta |L|\geq \frac{\delta(1-\epsilon-\beta_{out})}{2}N,
\end{equation*}
and
\begin{align*}
    |B(g)|\geq &|L\cap gRg^{-1}|=|L\setminus (gRg^{-1})^c|\geq |L\setminus gLg^{-1}|-|(gRg^{-1})^c\setminus gLg^{-1}|\\
    = & |L|-|L\cap gLg^{-1}|-|R^c\setminus L|=|L|-|L\cap gLg^{-1}|-|(L\cup R)^c|\\
    \geq & \frac{\delta(1-\epsilon-\beta_{out})}{2}N-(\epsilon+\beta_{out})N.
\end{align*}
Inserting the above two estimates into (\ref{eq:houtAgBg}) yields
\begin{equation*}
    h_{out}\leq \frac{2}{\frac{\delta(1-\epsilon-\beta_{out})}{2}-(\epsilon+\beta_{out})}\beta_{out}.
\end{equation*}
Let us collect our choices of $\delta$ and $\epsilon$. We have required that 
\[\delta<\frac{1}{3},\,\,\delta<1-\frac{2}{3(1-\epsilon-\beta_{out})},\,\,\text{and}\,\,\delta<\left(\frac{1}{2}-\epsilon-2\beta_{out}\right)^2\frac{2}{1+\epsilon}-\epsilon.\]
Since we have assumed $\beta_{out}\leq \epsilon$, it is enough to require that
\[\delta<\frac{1}{3},\,\,\delta<1-\frac{2}{3(1-2\epsilon)},\,\,\text{and}\,\,\delta<\left(\frac{1}{2}-3\epsilon\right)^2\frac{2}{1+\epsilon}-\epsilon.\]
To that end, we set $\epsilon=\frac{1}{100}$ and $\delta=\frac{1}{5}$. In conclusion, we have
\begin{equation*}
    \beta_{out}\geq \begin{cases}
        \epsilon h_{out}, & \text{if $\epsilon N\leq |Y|\leq \frac{N}{2}$,}\\
        h_{out}, & \text{if $|Y|\geq \frac{N}{2}$,}\\
        \frac{\epsilon}{2}h_{out}, & \text{if $\max\{|L|,|R|\}\geq\frac{1+\epsilon}{2}N$,}\\
        \frac{\epsilon}{2}h_{out}, & \text{if $\beta_{out}\geq \epsilon$,}\\
        \frac{\delta(1-2\epsilon)-4\epsilon}{4}h_{out}, & \text{otherwise.}
    \end{cases}
\end{equation*}
This completes the proof. 
\end{proof}

Next, we deal with vertex transitive graphs. 
\begin{theorem}\label{thm:vertextransitivegraph}
    Let $G=(V,E)$ be a non-bipartite finite vertex transitive graph.Then there exists an absolute constant $C$ such that
$$h_{out}\leq C\beta_{out}.$$
\end{theorem}

Our proof is built upon the proof of Moorman-Ralli-Tetali \cite[Theorem 2.4]{MRT} modified with key ideas from Biswas-Saha \cite{BS_2021} and a very recent preprint of Saha \cite{Saha}. 
Saha \cite{Saha} shows for vertex transitive graphs, (twisted) Cayley and (twisted) Cayley sum graphs the strong result that $\beta\geq C\frac{h}{d}$ for some absolute constant $C$, where $\beta$ is Trevisian's bipartiteness constant (recall (\ref{eq:Trevisan})) and $h$ is the Cheeger constant (recall (\ref{eq:Cheegerhhout})), both defined via the edge boundaries. This confirms a previous conjecture of Moorman, Ralli and Tetali. Note that $\beta\geq C\frac{h}{d}$ does not implies $\beta_{out}\geq Ch_{out}$, although both $\beta, \beta_{out}$ and $h,h_{out}$ are closely related.

\begin{proof}[Proof of Theorem \ref{thm:vertextransitivegraph}]
Let $\epsilon>0$ be a constant to be determined. We only need to consider the case that $\beta_{out}<\epsilon$.
Let $L$ and $R$ be the two subsets of $V$ achieving $\beta_{out}$. Define 
\[Y=\{v\in V\,|\, \mathrm{dist}_G(v,L\cup R)\geq 2\}.\]
We denote $N=|V|$. If $\epsilon N<|Y|\leq \frac{1}{2}N$, then $h_{out}<\frac{1}{\epsilon}\beta_{out}$; If $|Y|\geq \frac{1}{2}N$, then $h_{out}\leq \beta_{out}$. From now on, we assume $|Y|\leq \epsilon N$.

Let $X$ be a subgroup of $\mathrm{Aut}(G)$ which acts transitively on $V$ such that no proper subgroup of $X$ acts transitively on $V$. For any $g\in X$, we define the following two sets:
\[A(g)=(gL\cap L)\cup(gR\cap R),\,\,\text{and}\,\,B(g)=(gR\cap L)\cup(gL\cap R).\]
It is direct to check that $A(g)$ and $B(g)$ are disjoint, and we have via applying Lemma \ref{lemma:subsetabcd} 
\begin{equation*}
    h_{out}\leq \frac{2N}{\min\{|A(g)|,|B(g)|\}}\beta_{out}
\end{equation*}
Without loss of generality, we assume $|L|\geq |R|$. If $|L|\geq \frac{1+\epsilon}{2}N$, then we have $\beta_{out}\geq \epsilon$. Hence, we assume, from now on, that $|L|\leq \frac{1+\epsilon}{2}N$. One check also a lower bound $|L|\geq \frac{1-\epsilon-\beta_{out}}{2}N$.

Next, we show there exists a $g\in X$ such that 
\begin{equation}\label{eq:tobeproved}
    |L\cap gL|\in (\delta|L|, (1-\delta)|L|),
\end{equation}
for some constant $\delta>0$ to be determined.
We prove by contradiction. Assume no such $g\in X$ exist. Then we have a partition of $X$ given by 
\[X_1=\{g\in X\,|\,|L\cap gL|\geq (1-\delta)|L|\}\,\,\text{and}\,\,X_2=\{g\in X\,|\,|L\cap gL|\leq \delta|L|\}.\]
We then check that $X_1^2\subseteq X_1$ if $\delta<\frac{1}{3}$ and $X_2^2\subseteq X_1$ if $\delta<1-\frac{2}{3(1-\epsilon-\beta_{out})}$.

We next show $X_2$ is not empty. Assume to the contrary that $X_2=\emptyset$. Then, we have $X_1=X$. We calculate
\begin{equation*}
    \sum_{g\in X}|gL\cap L|=\sum_{g\in X}\sum_{u\in L:\, \exists\, v\in L \,\text{such that}\, gv=u}1=\sum_{u\in L}\sum_{g\in X: g^{-1}u\in L}1=\sum_{u\in L}\sum_{v\in L}|X_v|=t|L|^2,
\end{equation*}
where we denote $X_v$ the stabilizer subgroup of $v$ and $t=|X_v|$ stands for its size. This leads to
\begin{equation*}
    t|L|^2=\sum_{g\in X}|gL\cap L|\geq (1-\delta)|L||X|=(1-\delta)|L|tN,
\end{equation*}
and furthermore, $|L|\geq (1-\delta)|L|>\frac{2}{3}N$ if $\delta<\frac{1}{3}$.
This contradicts to $|L|\leq\frac{1+\epsilon}{2}N$ once we choose $\epsilon<\frac{1}{3}$. That is, $X_2$ can not be empty.

Then $X_1$ is a proper subgroup of $X$ with its complement $X_2$ satisfying $X_2^2\subseteq X_1$. Hence, $X_1$ is a subgroup of index $2$ and $X_2$ is its unique nontrivial coset. In particular, we have $|X_1|=|X_2|=\frac{tN}{2}$. Moreover, the action of $X_1$ on $V$ has two obits $O_1$ and $O_2$. Their sizes satisfy $|O_1|=|O_2|=\frac{N}{2}$. For $v\in O_1$, we check that
\[O_1=\{x_1v\,|\,x_1\in X_1\}\,\,\text{and}\,\,O_2=\{x_2v\,|\,x_2\in X_2\}.\]
Let us compute 
\begin{align*}
    \sum_{g\in X_1}|L\cap gL|=&\sum_{g\in X_1}\sum_{u\in L: \,\exists\,v\in L\,\text{such that}\,gv=u}1=\sum_{u\in L}\sum_{g\in X_1: g^{-1}u\in L}1\\
    =&\sum_{u\in L\cap O_1}\sum_{g\in X_1: g^{-1}u\in L}1+\sum_{u\in L\cap O_2}\sum_{g\in X_1: g^{-1}u\in L}1\\
    =&\sum_{u\in L\cap O_1}\sum_{g\in X_1: g^{-1}u\in L\cap O_1}1+\sum_{u\in L\cap O_2}\sum_{g\in X_1: g^{-1}u\in L\cap O_2}1\\
    =&\sum_{u\in L\cap O_1}\sum_{v\in L\cap O_1}|\{g\in X_1\,|\,g^{-1}u=v\}|+\sum_{u\in L\cap O_2}\sum_{v\in L\cap O_2}|\{g\in X_1\,|\,g^{-1}u=v\}|.
\end{align*}
We observe that for any given $u,v\in O_1$,
\[|\{g\in X_1\,|\,g^{-1}u=v\}|+|\{h\in X_2\,|\,h^{-1}u=v\}|=t,\]
and the set $\{h\in X_2\,|\,h^{-1}u=v\}=\emptyset$. That is, we have $|\{g\in X_1\,|\,g^{-1}u=v\}|=t$. Similarly, we derive for any $u,v\in O_2$ that $|\{h\in X_2\,|\,h^{-1}u=v\}|=t$. This gives that 
\[\sum_{g\in X_1}|L\cap gL|=t\left(|L\cap O_1|^2+|L\cap O_2|^2\right).\]
By definition of $X_1$ and the fact $|X_1|=\frac{tN}{2}$, we obtain
$\frac{1-\delta}{2}N|L|\leq |L|^2-2|L\cap O_1||L\cap O_2|$, and hence
\begin{equation*}
    |L\cap O_1||L\cap O_2|\leq \frac{\epsilon+\delta}{4}N|L|.
\end{equation*}

Therefore, there exists $i\in \{1,2\}$ such that $    |L\cap O_i|\leq \frac{1}{2}\sqrt{(\epsilon+\delta)N|L|}$.
Let $j\in \{1,2\}\setminus\{i\}$ be the other index. Then we have
\begin{equation}\label{eq:LcapOi}
 |L\cap O_j|\geq |L|-\frac{1}{2}\sqrt{(\epsilon+\delta)N|L|}.
\end{equation}
We start to argue that the graph $G$ is bipartite with the bi-partition given by $O_1$ and $O_2$. Here, we borrow a key idea from \cite[Lemma 2.2]{Saha}. 

We argue by contradiction. Suppose that either $O_i$ or $O_j$ contains a pair of adjacent vertices $u$ and $v$. If $u,v\in O_j$, we have $O_j=X_1u=X_1v$. If $u,v\in O_i$, we have $O_j=X_2u=X_2v$. Here, we use the notation $X_1u=\{x_1u\,|\,x_1\in X_1\}$.  Let us consider the induced subgraph of $O_j$ in $G$. Since either $O_j=X_1u=X_1v$ or $O_j=X_2u=X_2v$, every vertex in $O_j$ has at least one neighbor in the subgraph. That is, the subgraph of $O_j$ is non-trivial. Moreover, this subgraph of $O_j$ is regular. Indeed, for any $u,v\in O_j$, there exists an $h\in X_1$ such that $v=hu$. Therefore, we have 
\[|\{w\in O_j: w\sim u\}|=|\{h(w)\in O_j: h(w)\sim h(u)\}|=|\{w\in O_j: w\sim v\}|.\]
That is, every pair of vertices in $O_j$ have the same degree $m$ in the subgraph. So the normalized adjacency matrix $\frac{1}{m}A_{O_j}$ of the subgraph of $O_j$ is doubly stochastic. By Birkhoff-von Neumann theorem, there exists permutations $\theta_1,\ldots,\theta_m$ of $O_j$ such that the adjacency matrix $A_{O_j}=\sum_{k=1}^mP_{\theta_k}$, where $P_{\theta_k}$ stands for the permutation matrix of $\theta_k$. Indeed, we have $v$ and $\theta_k(v)$ are adjacent for any $v\in O_j$ and any $k\in \{1,\dots,m\}$. Moreover, the neighborhood $\{u\in O_j\,|\,u\sim v\}$ of any vertex $v\in O_j$ coincides with $\bigcup_{k=1}^m\{\theta_k(v)\}$.

For any $k\in \{1,\ldots,m\}$, we have 
\begin{equation}\label{eq:tobecontradict}
    |\theta_k(O_j)\cap O_j|=|O_j|=\frac{N}{2}.
\end{equation}
However, we have, on the other hand, 
\begin{align*}
    |\theta_k(O_j)\cap O_j|\leq &|\theta_k(O_j\cap L)\cap O_j|+|\theta_k(O_j\cap L^c)\cap O_j|\\
    =& |\theta_k(O_j\cap L)\cap O_j\cap L|+|\theta_k(O_j\cap L)\cap O_j\cap L^c|+|\theta_k(O_j\cap L^c)\cap O_j|\\
    \leq & I(L)+2|O_j\cap L^c|=I(L)+2(|O_j|-|L\cap O_j|)\\
    \leq& \beta_{out}N+N-2|L|+\sqrt{(\epsilon+\delta)N|L|}\\
    \leq &\left(\epsilon+2\beta_{out}+\sqrt{\frac{(\epsilon+\delta)(1+\epsilon)}{2}}\right)N,
\end{align*}
where the estimate (\ref{eq:LcapOi}) has been applied. Once we choose $\delta<\left(\frac{1}{2}-\epsilon-2\beta_{out}\right)^2\frac{2}{1+\epsilon}-\epsilon$, we have $|\theta_k(O_j)\cap O_j|<\frac{N}{2}$, which contradicts to the fact (\ref{eq:tobecontradict}). This proves that the graph $G$ is bipartite with the bi-partition $O_1$ and $O_2$. This contradicts to our assumption that $G$ is non-bipartite. Therefore, we finish the proof of (\ref{eq:tobeproved}).

Now, following the same line as in the proof of Theorem \ref{thm:cayleysumgraph}, we can estimate $|A(g)|$ and |B(g)| to conclude that 
\begin{equation*}
    h_{out}\leq \frac{2}{\frac{\delta(1-\epsilon-\beta_{out})}{2}-(\epsilon+\beta_{out})}\beta_{out}.
\end{equation*}
Recall that we have assumed $\beta_{out}<\epsilon$. Setting $\epsilon=\frac{1}{100}$ and $\delta=\frac{1}{5}$ fulfils all our requirements. This completes the proof. 
\end{proof}
\section{Extensions on connection graphs}\label{extensions}
Theorem \ref{thm:main} is extendable to a general \emph{connection graph}. Let $G=(V,E)$ be an undirected finite graph and $E^{or}:=\{(x,y),(y,x)\,|\,\{x,y\}\in E,\,x\neq y\}\cup \{(x,x)\,|\,\{x,x\}\in E\}$ be the set of oriented edges and self-loops. A \emph{connection} $\sigma: E^{or}\to O(k)$ or $U(k)$ is a map assigning to each oriented edge or self-loop an $k\times k$ orthogonal matrix or unitary matrix such that for each $\{x,y\}\in E$
\[\sigma_{xy}=\sigma_{yx}^{-1}.\]
Particularly, for the connection $\sigma_{xx}$ of a self-loop at $x$, we have $\sigma_{xx}=\sigma_{xx}^{-1}$.
The couple $(G,\sigma)$ is called a connection graph. A connection graph with an $O(1)$-connection is simply a signed graph. 

Let $\mathbb{K}=\mathbb{R}$ if we have an $O(k)$-connection and $\mathbb{K}=\mathbb{C}$ if we have an $U(k)$-connection. For a vector-valued function $f: V\to \mathbb{K}^k$, the \emph{normalized connection Laplacian} $\Delta^\sigma$ \cite{SW} is defined as follows:
\begin{equation*}
    \Delta^\sigma f(x)=\frac{1}{d_x}\sum_{y:y\sim x}(f(x)-\sigma_{xy}f(y)).
\end{equation*}
The operator $\Delta^\sigma$ has $k|V|$ real eigenvalues (counting multiplicities) lying in the interval $[0,2]$. Let us denote the smallest eigenvalue by $\lambda_1^{\sigma}$. Cheeger-type eigenvalue estimates for graph connection Laplacian have been studied in \cite{BSS}, while Buser-type eigenvalue estimates have been studied in \cite{LMP}. We point out that in the case of $U(1)$-connection, the corresponding connection Laplacian is known as the \emph{magnetic Laplacian}, for which the Cheeger-type estimate and its higher order versions have been established in \cite{LLPP}.

Let $\pi:V\to \mathbb{R}_{>0}$ be a vertex measure of $(G,\sigma)$. 
We extend the definition of the constant $\lambda^{\sigma}_{\infty}$ to be the optimal constant in the following inequality:
\begin{equation*}
\lambda_{\infty}^{\sigma}\sum_{x\in V}\Vert f(x)\Vert^{2}\pi(x)\leq\sum_{x\in V}\sup_{y:y\sim x}\Vert f(x)-\sigma_{xy}f(y)\Vert^{2}\pi(x),
\end{equation*}
where $f: V\to \mathbb{K}^k$ is arbitrary, and $\Vert\cdot\Vert$ stands for the Euclidean or Hermitian norm of vectors in $\mathbb{K}^k$. For a $d$-regular connection graph, it holds that
\begin{equation*}
    2\lambda_1^\sigma\leq \lambda_\infty^\sigma\leq 2d\lambda_1^\sigma.
\end{equation*}

Let us denote by $\mathbb{S}^{k-1}=\{v\in \mathbb{K}^k\,|\,\Vert v\Vert=1\}$. Then, for a subset $V_1\subseteq V$, we define the $\infty$-$\mathbb{S}^{k-1}$-frustration constant $\eta^{\sigma,\pi}_{\infty}(V_1)$ as 
\begin{equation*}
\eta^{\sigma,\pi}_{\infty}(V_{1})=\min_{\tau:V_{1}\rightarrow \mathbb{S}^{k-1}}\frac{1}{2}\sum_{x\in V_{1}}\sup_{\substack{y:y\sim x\\ y\in V_1}}\Vert \tau(x)-\sigma_{xy}\tau(y)\Vert\pi(x),
\end{equation*}
Correspondingly, we define 
\begin{equation*}
    \eta^*_{out}:=\min_{\emptyset\neq V_{1}\subseteq V}\frac{2\eta^{\sigma,\pi}_{\infty}(V_{1})+\pi(\partial_{out}(V_{1}))}{\pi(V_{1})},
\end{equation*}
and 
\begin{equation*}
    \eta^*_S:=\inf_{\emptyset\neq V_{1}\subseteq V}\frac{2\eta^{\sigma,\pi}_{\infty}(V_{1}^{\circ})+\pi(\partial_{S}(V_{1}))}{\pi(V_{1})}.
\end{equation*}
\begin{remark}
    We are following here the terminology and notations from the work of Bandeira, Singer and Spielman \cite{BSS}. We define $p$-$\mathbb{S}^{k-1}$-frustration constant as 
\begin{equation*}
\eta^{\sigma,\pi}_{p}(V_{1})=\min_{\tau:V_{1}\rightarrow \mathbb{S}^{k-1}}\frac{1}{2}\sum_{x\in V_{1}}\sum_{\substack{y:y\sim x
\\y\in V_1}}\Vert \tau(x)-\sigma_{xy}\tau(y)\Vert^p\pi(x).
\end{equation*}
Then the constant $\eta^{\sigma,\pi}_{2}(V_{1})$ reduce to the $\mathbb{S}^{k-1}$-frustration index constant in \cite[(1.5)]{BSS}. Our constants $\eta^*_{out}$ and $\eta^*_S$ should be considered as vertex boundary versions of the \emph{partial $\mathbb{S}^{k-1}$-frustration constant} $\eta^*$ \cite[(2.1)]{BSS} given by 
\[\eta^*:=\min_{\emptyset\neq V_{1}\subseteq V}\frac{\eta^{\sigma,\pi}_{2}(V_{1})+|\partial(V_{1})|}{\mathrm{vol}(V_{1})}.\] Indeed, it is shown in \cite[Theorem 2.2]{BSS} that 
$(\eta^*)^2/10\leq\lambda_1^\sigma\leq \eta^*$.
\end{remark}
A generalization of Theorem \ref{thm:main} is given below.
\begin{theorem}\label{thm:main_connection}
For a connection graph $(G,\sigma)$ with a vertex measure $\pi$, we have 
\begin{align}
\lambda^{\sigma}_{\infty}\geq \left(\sqrt{1+\frac{2}{\sqrt{5}}\eta^*_{out}}-1\right)^2,
\end{align}
and
\begin{align}
\lambda^{\sigma}_{\infty}\geq \left(\sqrt{1+\frac{2}{\sqrt{5}}\eta^*_{S}}-1\right)^2.
\end{align}
\end{theorem}
For the proof, we need the following lemma from \cite[(3.1) and Appendeix A]{BSS}.
\begin{lemma}\label{lem:BSS}
For any $t\in [0,1]$, we define $Y_{t}:\overline{B_{1}(0)}:=\{v\in\mathbb{K}^{k}\,|\,\Vert v\Vert \leq 1\}\rightarrow\mathbb{S}^{k-1}\cup\{0\}$ as
$$ Y_{t}(z):=\begin{cases}
    \frac{z}{\Vert z\Vert}, & \text{if $\Vert z\Vert \geq t$,}\\
    0, & \text{otherwise,}
\end{cases}
$$
where we use the notation that $z/\Vert z\Vert$ be an arbitrarily chosen unit vector when $\Vert z\Vert=0$. Then, we have for any $z_{1}$, $z_{2}\in \mathbb{K}^{k}$, 
$$\int_{0}^{1}\left\Vert Y_{\sqrt{t}}(z_{1})-Y_{\sqrt{t}}(z_{2})\right\Vert dt\leq\frac{\sqrt{5}}{2}(\Vert z_{1}-z_{2}\Vert )(\Vert z_{1}\Vert+\Vert z_{2}\Vert).$$
\end{lemma}
\begin{proof}[Proof of Theorem \ref{thm:main_connection}]
Without loss of generality, we assume that $\max_{x\in V}\Vert f(x)\Vert=1$. Then, for any given $t\in [0,1]$, we define the subset 
$$V^{f}(\sqrt{t})=\left\{x\in V\left|\,\Vert f(x)\Vert\geq\sqrt{t}\right.\right\}.$$
Let $Y_t$ be defined as in Lemma \ref{lem:BSS}. Then the proof follows along the same line as the proof of Theorem \ref{thm:main} by considering the quantity
\begin{equation}
    \frac{\int_0^1\sum_{x\in V}\sup_{y:y\sim x}\left\Vert Y_{\sqrt{t}}(f(x))-\sigma_{xy}Y_{\sqrt{t}}(f(y))\right\Vert\pi(x)dt}{\int_0^1\pi\left(V^f(\sqrt{t})\right)dt}
\end{equation}
and applying Lemma \ref{lem:BSS} instead of Lemma \ref{lemma:ineq c}. We omit the details here.
\end{proof}

To conclude this section, we discuss another extension to a connection graph $(G,\sigma)$ with $k$-cyclic group connections. That is, 
we have $\sigma:E^{or}\rightarrow \mathcal{S}_k:=\{\xi^{0},\xi^{1},\dots,\xi^{k-1}\}$, where $\xi^{i},i=0,1,\dots,k-1$ are the $k$-th primitive roots of unity. We define the constant $\lambda^{\sigma}_{\infty}$ as the optimal constant such that 
$$\lambda^{\sigma}_{\infty}\sum_{x\in V}|f(x)|^{2}\pi(x)\leq\sum_{x\in V}\sup_{y:y\sim x}|f(x)-\sigma_{xy}f(y)|^{2}\pi(x),$$
holds for any function $f:V\rightarrow\mathbb{C}$,
where $|\cdot|$ stands for the norm of a complex number. Let us define the two constants
\begin{equation*}
    h^\sigma_{out}:=\min_{\emptyset\neq V_{1}\subseteq V}\frac{2\iota^{\sigma,\pi}_{\infty}(V_{1})+\pi(\partial_{out}(V_{1}))}{\pi(V_{1})},\,\,\text{and}\,\,
    h^\sigma_S:=\inf_{\emptyset\neq V_{1}\subseteq V}\frac{2\iota^{\sigma,\pi}_{\infty}(V_{1}^{\circ})+\pi(\partial_{S}(V_{1}))}{\pi(V_{1})}.
\end{equation*}
with 
\begin{equation*}
    \iota^{\sigma,\pi}_{\infty}(V_{1})=\min_{\tau:V_1\to \mathcal{S}_k}\frac{1}{2}\sum_{x\in V_1}\sup_{\substack{y:y\sim x\\y\in V_1}}|\tau(x)-\sigma_{xy}\tau(y)|\pi(x),
\end{equation*} for any subset $V_1\subseteq V$.
Then we have the following result.
\begin{theorem}\label{thm:main_cyclic}
For a connection graph $(G,\sigma)$ with a $k$-cyclic group connection $\sigma:V\rightarrow \mathcal{S}_k$, we have the following inequality
$$\lambda^{\sigma}_{\infty}\geq \left(\sqrt{1+\frac{1}{2}h^\sigma_{out}}-1\right)^2.$$
\end{theorem}
For the proof, we replace $Y_{t}$ by another indicator function $Y_{t,\theta}$: For $\theta\in[0,2\pi)$ and $k\in \mathbb{N}$, define $k$ disjoint sectorial regions
$$Q^{\theta}_{j}:=\left\{re^{i\alpha}\in \overline{B_{1}(0)}\,\left|\,r\in(0,1],\,\alpha\in\left[\theta+\frac{2\pi j}{k},\theta+\frac{2\pi(j+1)}{k}\right)\right.\right\},$$
where $j=0,1,\dots,k-1$. Then, for any $t\in(0,1],$ define $Y_{t,\theta}:\overline{B_{1}(0)}\rightarrow\mathbb{C}$ as
$$Y_{t,\theta}(z)=\left\{
\begin{aligned}
\xi^{j},\textrm{ if }z\in Q^{\theta}_{j}\setminus B_{t}(0),\\
0,\quad\quad\,\,\textrm{ if }z\in B_{t}(0).
\end{aligned}
\right.
$$
The following lemma is taken from \cite[Lemma 4.2]{LLPP}.
\begin{lemma}\label{lem:LLPP}
For any two points $z_{1}$, $z_{2}\in\overline{B_{1}(0)}$, we have
$$\frac{1}{2\pi}\int_{0}^{2\pi}\int_{0}^{1}\left|Y_{\sqrt{t},\theta}(z_{1})-Y_{\sqrt{t},\theta}(z_{2})\right|dtd\theta\leq2|z_{1}-z_{2}|(|z_{1}|+|z_{2}|).$$
\end{lemma}
The proof of Theorem \ref{thm:main_cyclic} follows along the same line as the proof of Theorem \ref{thm:main} via using the Lemma \ref{lem:LLPP}  instead of Lemma \ref{lemma:ineq c}.

\section*{Acknowledgement}
We thank the referees for their helpful comments and suggestions. SL thanks Tuan Tran for helpful discussions. This work is supported by the National Key R and D Program of China 2020YFA0713100, the National Natural Science Foundation of China (No. 12031017), and Innovation Program for Quantum Science and Technology 2021ZD0302902.


\begin{thebibliography}{99}
\bibitem{Alon} N. Alon, Eigenvalues and expanders, Combinatorica 6 (1986), no. 2, 83-96.

\bibitem{Alon2} N. Alon, Large sets in finite fields are sumsets, J. Number Theory 126 (2007), no. 1, 110–118.

\bibitem{AM} N. Alon and V. Milman, $\lambda_{1}$, isoperimetric inequalities for graphs, and superconcentrators, J. Combin. Theory Ser. B 38 (1985), no. 1, 73-88.

\bibitem{AT} M. Amooshahi and B. Taeri, On Cayley sum graphs of non-abelian groups, Graphs Combin. 32 (2016), no. 1, 17-29.

\bibitem{AL} F. M. Atay and S. Liu, Cheeger constants, structural balance, and spectral clustering analysis for signed graphs, Discrete Math. 343 (2020), no. 1, 111616, 26 pp.

\bibitem{BSS} A. S. Bandeira, A. Singer and D. A. Spielman, A Cheeger inequality for the graph connection Laplacian, SIAM J. Matrix Anal. Appl. 34 (2013), no. 4, 1611-1630.

\bibitem{BJ} F. Bauer and J. Jost, Bipartite and neighborhood graphs and the spectrum of the normalized graph Laplace operator, Comm. Anal. Geom. 21 (2013), no. 4, 787-845.

\bibitem{BKW} F. Bauer, M. Keller and R. K. Wojciechowski, Cheeger inequalities for unbounded graph Laplacians, J. Eur. Math. Soc. (JEMS) 17 (2015), no. 2, 259-271.

\bibitem{BCKW} F. Belardo, S. M. Cioabă, J. Koolen and J. Wang, Open problems in the spectral theory of signed graphs, Art Discrete Appl. Math. 1 (2018), no. 2, Paper No. 2.10, 23 pp.

\bibitem{Biggs} N. Biggs, Algebraic graph theory, Second edition, Cambridge Mathematical Library, Cambridge University Press, Cambridge, 1993. 

\bibitem{Biswas_EJC} A. Biswas, On a Cheeger type inequality in Cayley graphs of finite groups, European J. Combin. 81 (2019), 298-308.

\bibitem{BS_AC} A. Biswas and J. P. Saha, A Cheeger type inequality in finite Cayley sum graphs, Algebr. Comb. 4 (2021), no. 3, 517-531.

\bibitem{BS_2021} A. Biswas and J. P. Saha, A spectral bound for vertex-transitive graphs and their spanning subgraphs, Algebr. Comb.6 (2023), no. 3, 689-706.

\bibitem{BHT} S. Bobkov, C. Houdr\'e and P. Tetali, $\lambda_{\infty}$, vertex isoperimetry and concentration, Combinatorica 20 (2000), no. 2, 153-172.

\bibitem{BM} J. A. Bondy and U. S. R. Murty, Graph theory, Grad. Texts in Math., 244, Springer, New York, 2008.

\bibitem{BGKLM} M. Bonnefont, S. Golénia, M. Keller, S. Liu and F. Münch, Magnetic-sparseness and Schrödinger operators on graphs, Ann. Henri Poincaré 21 (2020), no. 5, 1489–1516.

\bibitem{BGGT} E. Breuillard, B. Green, R. Guralnick and T. Tao,  Expansion in finite simple groups of Lie type, J. Eur. Math. Soc. (JEMS) 17 (2015), no. 6, 1367-1434.

\bibitem{Cheeger} J. Cheeger, A lower bound for the smallest eigenvalue of the Laplacian, Problems in analysis (Sympos. in honor of Salomon Bochner, Princeton Univ., Princeton, N.J., 1969), pp. 195-199. Princeton Univ. Press, Princeton, N.J., 1970.

\bibitem{CGW}B. Cheyne, V. Gupta and C. Wheeler, Hamilton cycles in addition graphs, Rose-Hulman Undergrad. Math J. 4 (2003), no. 1, 1-17.

\bibitem{Chung2}F.R.K. Chung, Diameters and eigenvalues, J. Amer. Math. Soc. 2 (1989), no. 2, 187–196.

\bibitem{Chung} F. R. K. Chung, Spectral graph theory, CBMS Regional Conference Series in Mathematics, 92, American Mathematical Society, Providence, RI, 1997.

\bibitem{D} J. Dodziuk, Difference equations, isoperimetric inequality and transience of certain random walks, Trans. Amer. Math. Soc. 284 (1984), no. 2, 787-794.

\bibitem{Green} B. Green, Counting sets with small sumset, and the clique number of random Cayley graphs, Combinatorica 25 (2005), no. 3, 307-326.  


\bibitem{GLS}D. Grynkiewicz, V.F. Lev and O. Serra, Connectivity of addition Cayley graphs, J. Combin. Theory Ser. B 99 (2009), no. 1, 202–217.

\bibitem{Harary} F. Harary, On the notion of balance of a signed graph, Michigan Math. J. 2(1953/1954), no. 2, 143-146.

\bibitem{HS} C. Houdr\'e and T. Stoyanov, Expansion and isoperimetric constants for product graphs, Combinatorica 26 (2006), no. 4, 455-473.

\bibitem{HL} C. Hu and S. Liu, 
Discrete Bakry-\'Emery curvature tensors and matrices of connection graphs, arXiv:2209.10762.

\bibitem{KKRT}  B. Klartag, G. Kozma, P. Ralli and P. Tetali, Discrete curvature and abelian groups, Canad. J. Math. 68 (2016), no. 3, 655–674.

\bibitem{LLPP} C. Lange, S. Liu, N. Peyerimhoff and O. Post, Frustration index and Cheeger inequalities for discrete and continous magenatic Laplacians, Calc. Var. Partial Differential Equations 54 (2015), no. 4, 4165-4196.

\bibitem{LS} G. F. Lawler and A. D. Sokal, Bounds on the $L^2$ spectrum for Markov chains and Markov processes: a generalization of Cheeger's inequality, Trans. Amer. Math. Soc. 309 (1988), no. 2, 557-580.

\bibitem{LOT} J. Lee, S. Oveis Gharan and L. Trevisan, Multi-way spectral partitioning and higher-order Cheeger inequalities, J. ACM 61 (2014), no. 6, Art. 37, 30 pp.

\bibitem{Liu} S. Liu, Multi-way dual Cheeger constants and spectral bounds of graphs, Adv. Math. 268 (2015), 306-338.

\bibitem{LMP} S. Liu, F. M\"unch and N. Peyerimhoff, Curvature and higher order Buser inequalities for the graph connection Laplacian,
SIAM J. Discrete Math. 33 (2019), no.1, 257-305.

\bibitem{Miclo} L. Miclo, On eigenfunctions of Markov processes on trees, Probab. Theory Related Fields 142 (2008), no. 3-4, 561-594.

\bibitem{Miclo2} L. Miclo, On hyperboundedness and spectrum of Markov operators, Invent. Math. 200 (2015), no. 1, 311-343.

\bibitem{Mohar} B. Mohar, Isoperimetric numbers of graphs, J. Combin. Theory Ser. B 47 (1989), no. 3, 274-291.

\bibitem{MRT} N. Moorman, P. Ralli and P. Tetali, On the bipartiteness constant and expansion of Cayley graphs, European J. Combin. 102 (2022), Paper No. 103481, 12 pp.

\bibitem{Saha} J. P. Saha, A Cheeger inequality for the lower spectral gap, arXiv:2306.04436, 2023.

\bibitem{SJ} A. Sinclair and M. Jerrum, Approximate counting, uniform generation and rapidly mixing Markov chains, Inform. and Comput. 82 (1989), no. 1, 93-133.

\bibitem{SW} A. Singer and H.-T. Wu, Vector diffiusion maps and the connection Laplacian, Comm. Pure Appl. Math. 65 (2012), no. 8, 1067-1144.

\bibitem{Trevisan} L. Trevisan, Max cut and the smallest eigenvalue, SIAM J. Comput. 41 (2012), no. 6, 1769-1786.

\bibitem{Z} T. Zaslavsky, Signed graphs, Discrete Appl. Math. 4 (1982), no. 1, 47-74.

\bibitem{ZY} J.-Y. Zhang and Y.-Y. Yang, 
On the subgraphs of Cayley sum graphs, Discrete Math. 346 (2023), no.8, Paper No. 113403, 7 pp.

\end{thebibliography}
\end{document}